\documentclass[preprint,3p]{elsarticle}
\usepackage{lineno,hyperref}
\modulolinenumbers[5]
\usepackage{amsmath,amssymb,amsthm}
\usepackage{mathrsfs} %
\usepackage{graphicx} %
\usepackage{geometry} %
\usepackage{subfig}
\usepackage{epsfig}
\usepackage{float}
\usepackage{verbatim}
\usepackage{color}
\usepackage{breqn}
\usepackage{booktabs} 
\usepackage{multirow} 

\hypersetup{hidelinks}
\numberwithin{equation}{section}

\newtheorem{lemma}{Lemma}[section]
\newtheorem{theorem}{Theorem}[section]
\newtheorem{remark}{Remark}

\nonstopmode
\makeatletter
\newcommand{\Rmnum}[1]{\expandafter\@slowromancap\romannumeral #1@}
\makeatother

\begin{document}
\begin{frontmatter}
\def\ll{\mbox{ } \hskip 1em}
\def\matrix{\,\vcenter\bgroup\plainLet@\plainvspace@
    \normalbaselines
   \math\ialign\bgroup\hfil$##$\hfil&&\quad\hfil$##$\hfil\crcr
      \mathstrut\crcr\noalign{\kern-\baselineskip}}
\def\mathbb{}

\title{Pointwise error estimates of compact difference scheme for mixed-type time-fractional Burgers' equation}

\author[mymainaddress1]{Xiangyi Peng}
\ead{pengxy202207@163.com}
\author[mymainaddress1]{Da Xu}
\ead{daxu@hunnu.edu.cn}
\author[mymainaddress1]{Wenlin Qiu\corref{mycorrespondingauthor}}
\ead{qwllkx12379@163.com}

\cortext[mycorrespondingauthor]{Corresponding author. The work was supported by National Natural Science Foundation of China (No. 12071127).}
\address[mymainaddress1]{MOE-LCSM, School of Mathematics and Statistics, Hunan Normal University, Changsha, Hunan 410081, China}

\begin{abstract}
 In this paper, based on the developed nonlinear fourth-order operator and method of order reduction, a novel fourth-order compact difference scheme is constructed for the mixed-type time-fractional Burgers' equation, from which $L_1$-discretization formula is employed to deal with the terms of fractional derivative, and the nonlinear convection term is discretized by nonlinear compact difference operator. Then a fully discrete compact difference scheme can be established by  approximating spatial second-order derivative with classic compact difference formula. The convergence and stability are rigorously proved in the $L^{\infty}$-norm by the energy argument and mathematical induction. Finally, several numerical experiments are provided to verify the theoretical analysis. 
\end{abstract}


\end{frontmatter}

\section{Introduction}
\vskip 0.2mm
In this article, considering the following mixed-type time-fractional Burgers' equation
\begin{equation}\label{eq1.1}
\begin{array}{ll}
\mu_1\mathbb{D}_{t}^{(\alpha+1)}u+\mu_2\mathbb{D}_{t}^{(\alpha)}u+uu_x=\lambda u_{xx}(x,t), \qquad  (x,t)\in (0,L)\times (0,T],\quad \alpha\in (0, 1),
\end{array}
\end{equation}
with the initial conditions and the boundary conditions
\begin{equation}\label{eq1.2}
u(x,0)=\varphi_1 (x),\quad u_t(x,0)=\varphi_2 (x), \qquad x \in [0,L],
\end{equation}
\begin{equation}\label{eq1.3}
\begin{array}{ll}
u(0,t)=u(L,t)=0, \quad t \in [0, T],
\end{array}
\end{equation}
where $\lambda >0$ is the coefficient of kinematic viscosity, $\mu_1,\mu_2 \geq 0$ ($\mu_1^2+\mu_2^2\neq 0$), $\varphi_1(x)$ and $\varphi_2(x)$ are given functions, and the notation $\mathbb{D}_{t}^{(\beta)}u $ denotes the Caputo fractional derivative \cite{Fra_df1,Fra_df2}, defined by
\begin{equation*}
\mathbb{D}_{t}^{(\beta)}u(x,t) :=
 \begin{cases}
   \frac{1}{\Gamma(1-\beta)} \int_{0}^{t}\frac{\partial u(x,s)}{\partial s}(t-s)^{-\beta}ds,  &\quad 0 < \beta <1, \\
   \\
  \frac{\partial u(x,t)}{\partial t}, &\quad  \beta = 1,\\
  \\
  \frac{1}{\Gamma(2-\beta)} \int_{0}^{t}\frac{\partial^2 u(x,s)}{\partial s^2}(t-s)^{(1-\beta)}ds,  &\quad 1 < \beta <2, \\
 \end{cases}
\end{equation*}
where $\Gamma(\cdot)$ indicates the Gamma function.
\vskip 0.2mm
Burgers-type equations are the basic partial differential equations of applied mathematics, which are widely used in various fields, such as fluid mechanics \cite{Bur_yy11,Bur_yy22}, nonlinear acoustics \cite{Bur_yy21,BL}, gas dynamics \cite{Bur_yy31,Bur_yy12}, etc. And they can be used as a reference for solving complex problems such as Navier-Stokes equations. With the deepening of physical research, in past decades years, fractional Burgers-type equations had been proposed and studied, because fractional derivative can be used to describe the cumulative effect of wall friction through boundary layer \cite{app1,app2,app4}.
\vskip 0.2mm
Generally speaking, it is difficult to get the analytical solutions of partial differential equations, let alone fractional partial differential equations (FPDEs). Thus this drives people to search for high-precision and efficient numerical methods to solve FPDEs. In recent several years, a lot of studies \cite{burg_xyj2,LDF,BL,qsx2} for the time FPDEs and some profound results have been developed. As a classical PDE problem, time-fractional Burgers' equation is concerned and studied by many scholars naturally \cite{LCP,LLL,zhangqf}. For examples, the following time-fractional Burgers' equation have been studied by many scholars
\begin{equation}\label{eq1.4}
  \mathbb{D}_{t}^{(\alpha)}u+uu_x=\lambda u_{xx}, \qquad  (x,t)\in (0,L)\times (0,T],\quad \alpha\in (0, 1),
\end{equation}
which is obtained by taking $\mu_1=0$ and $\mu_2=1$ in (\ref*{eq1.1}). Generally, (\ref*{eq1.4}) belongs to parabolic time-fractional Burgers' equation. Qiu et al. \cite{qsx} considered an implicit difference scheme for one-dimensional time fractional Burgers' equation and proposed a novel iterative algorithm to implement it. Li and Li \cite{LLL} investigated the exact and numerical solutions of the time fractional Burgers' equation by using Cole-Hopf transformation,  the method of variables separation, the L1-scheme on graded meshes, and the Legendre-Galerkin spectral method. Akram \cite{burg_xyj1} developed a finite difference scheme which depended on a new approximation based on an extended cubic B-spline. And Zhang et al. \cite{zhangqf} developed fourth-order compact difference scheme. Besides, some articles have considered the following time-fractional Burgers' equation with hyperbolic properties
\begin{equation}\label{eq1.5}
  \mathbb{D}_{t}^{(\gamma)}u+uu_x=\lambda u_{xx}, \qquad  (x,t)\in (0,L)\times (0,T],\quad \gamma\in (1, 2),
\end{equation}
which is equivalent to the case of $\mu_1=1$ and $\mu_2=0$ in (\ref*{eq1.1}). Vong and Lyu \cite{vong} proposed a second-order linearized scheme for (\ref*{eq1.5}) in sense of maximum-norm. Furthermore, based on research of \cite{qsx}, Zhang \cite{zhangw} developed a semi-implicit finite difference scheme for the multi-term time-fractional Burgers-type equations.
\vskip 0.2mm
Up to now, many scholars have done a lot of work to the above two kinds of time-fractional Burgers' equations. However, researches on the numerical solutions of the mixed-type time-fractional Burgers' equations are still scant because of the complexity of numerical computation and the difficulty of theoretical analysis. Moreover, in certain physical phenomenon \cite{jinb}, mixed-type time-fractional derivative describes physical models more profound than single time-fractional derivative, which means that it is quite significant to investigate mixed-type time-fractional Burgers' equation. Inspired by some work \cite{wangxp,zhangqf2,zhangqf3,zhangqf4,zhangqf}, especially in dealing with nonlinear convection term, they ingeniously constructed fourth-order nonlinear compact operators. These encourage us to carry out the following work. The main aim of this article is to establish a fourth-order compact difference scheme for mixed-type time-fractional Burgers' equation. In this paper, we deal with the Caputo fractional derivatives by the $L_1$-discretization formula and Crank-Nicolson technique, and introduce nonlinear compact operator to treat the nonlinear convection term. The constructed compact difference scheme is stable and convergent with the convergence order of $4$ for space and $2-\alpha$ for time, which is verified by strict theoretical analysis.
\vskip 0.2mm
The rest of this paper is organized as follows. In section 2, the main is some preliminaries including grid division, some notations, and useful lemmas. Then, in section 3, establishing the compact difference scheme for the problem (\ref*{eq1.1})-(\ref*{eq1.3}). In section 4, the convergence and stability of the compact difference scheme are discussed via the discrete energy method and mathematical induction. In section 5, three numerical experiments are carried out to vaildate our theoretical analysis. Finally, a brief conclusion is given in section 6.
   \section{Preliminaries}
   \vskip 0.2mm
      In order to solve the problem (\ref*{eq1.1})-(\ref*{eq1.3}), first we divide the domain $[0,L]\times[0,T]$. Let $\omega_h:=\{x_i|0\leq i \leq M\}$ and $\omega_{\tau}:=\{t_n|0\leq n \leq N\}$ be two uniform meshes, where $x_i:=ih, h:=\frac{L}{M}$; $t_n:=n\tau$, $\tau:=\frac{T}{N}$, $M$ and $N$ are two given positive integers. Denote $\omega_{h\tau}:=\omega_h \times \omega_{\tau}$. For any grid function $v:=\{v_i^n|0\leq i \leq M,0\leq n \leq N\}$ defined on $\omega_{h\tau}$, introduce the following notations
      \begin{equation*}
        \begin{array}{ccc}
          v_i^{n-\frac{1}{2}}:=\frac{1}{2}(v_i^n+v_i^{n-1}),\qquad \delta_t v_i^{n-\frac{1}{2}}:=\frac{1}{\tau}(v_i^n-v_i^{n-1}),\qquad \delta_x v_{i-\frac{1}{2}}^n:=\frac{1}{h}(v_i^n-v_{i-1}^n),\\
          \\
          \Delta_x v_i^n:=\frac{1}{2h}(v_{i+1}^n-v_{i-1}^n),\qquad \delta_x^2 v_{i}^n:=\frac{1}{h}(\delta_x v_{i+\frac{1}{2}}^n-\delta_x v_{i-\frac{1}{2}}^n).
        \end{array}
        \end{equation*}
   \vskip 0.2mm
    And let $\Omega_h:=\{v|v=(v_0,v_1,\cdots,v_M )\}$ and $\mathring{\Omega}_h:=\{v|v\in\Omega_h, v_0=v_M=0 \}$ be the spaces of grid functions on $\omega_h$. For any $u,v \in \mathring{\Omega}_h$, we define the following inner products and norms
    \begin{equation*}
      \begin{array}{cc}
        \langle u,v \rangle:=h\sum\limits_{i=1}^{M-1}u_i v_i, \qquad (u,v):=h\sum\limits_{i=1}^{M}(\delta_x u_{i-\frac{1}{2}})(\delta_x v_{i-\frac{1}{2}}),\\
        \\
        \|v\|:=\sqrt{\langle v,v \rangle},\qquad |v|_1:=\sqrt{(v,v)},\qquad \|v\|_{\infty }:=\max \limits_{0\leq i \leq M}|v_i|.
      \end{array}
    \end{equation*}
   In addition, in order to discretize nonlinear term $uu_x$, we introduce the function $\psi$ \cite{gby,wangxp} as follows
   \begin{equation*}
    \psi(u_i,v_i):=\frac{1}{3}[u_i \Delta_x v_i+\Delta_x(u_i v_i)],\quad 1 \leq i \leq M-1.
   \end{equation*}
   \begin{remark}
    It is easy to know that $\psi$ is a bilinear function without commutativity.
  \end{remark}
  Finally, we introduce some useful lemmas which will be used later.
  \begin{lemma}\label{lem4.1}
    \cite{sunzz_book} For any grid functions $v,w \in \mathring{\Omega}_h$, we have
    \begin{equation}\label{eqn4.1}
      \langle w,\delta_x^2 v\rangle=-\langle \delta_x w,\delta_x v\rangle=\langle \delta_x^2 w,v\rangle.
    \end{equation}
   \end{lemma} 
   \begin{lemma} \label{lemma2.1}
    \cite{sunzz_book} For any grid function $v \in \mathring{\Omega}_h$, we have 
    \begin{equation}\label{eq2.1}
      \|v\|_{\infty }\leq \frac{\sqrt{L}}{2}|v|_1,\quad \|v\|\leq \frac{L}{\sqrt{6}}|v|_1,\quad \|\Delta_x v\| \leq |v|_1,\quad |v|_1 \leq \frac{2}{h}\|v\|.
    \end{equation}
   \end{lemma}
   \begin{lemma}\label{lemma2.2}
    \cite{sunzz} Suppose $G(t) \in C^2[0,t_n]$, existing $\tilde{C}>0$, it holds that
    \begin{equation}\label{eq2.2}
      \begin{split}
        &\left|\frac{1}{\Gamma(1-\alpha)}\int_{0}^{t_n}\frac{G'(s)}{(t_n-s)^{\alpha}} \,ds -\frac{\tau^{-1}}{\Gamma(1-\alpha)}\mathcal{D}_t^{(\alpha)}G(t_n)  \right| \leq \tilde{C} \tau^{2-\alpha}, \quad 0<\alpha<1,\\
        &\mathcal{D}_t^{(\alpha)}G(t_n):=b_0G(t_n)-\sum\limits_{i=1}^{n-1}(b_{n-i-1}-b_{n-i})G(t_i)-b_{n-1}G(t_0),
      \end{split}
    \end{equation}
    where $b_i:=\int_{t_i}^{t_{i+1}}s^{-\alpha} \,ds=\frac{\tau^{1-\alpha}}{1-\alpha} [(i+1)^{1-\alpha}-i^{1-\alpha}]$, $i\geq 0$. 
   \end{lemma}
   \begin{lemma}\label{lemma2.3}
    \cite{wangxp} Let $f(x)\in C^5[x_{i-1},x_{i+1}]$ and $F(x):=f''(x)$, then we have 
    \begin{equation}\label{eq2.3}
      f(x_i)f'(x_i)=\psi(f_i,f_i)-\frac{h^2}{2}\psi(F_i,f_i)+O(h^4).
    \end{equation}
   \end{lemma}
   \begin{lemma}\label{lemma2.4}
    \cite{sunh} For any grid functions $u \in \Omega_h$ and $v \in \mathring{\Omega}_h$, then we get
    \begin{equation}\label{eq2.4}
      \langle \psi(u,v),v\rangle=0.
    \end{equation}
   \end{lemma}
   \begin{lemma}\label{lemma2.5}
    \cite{wangxp} For any $w,u \in \mathring{\Omega}_h$ and $R\in{\Omega}_h$ satisfying 
    $$w_i=\delta_x^2 u_i-\frac{h^2}{12}\delta_x^2 w_i+R_i,\quad 1\leq i \leq M-1,$$
    we have
    \begin{equation}\label{eq2.5}
     \langle w,u \rangle=-|u|_1^2-\frac{h^2}{12}\|w\|^2+\frac{h^4}{144}|w|_1^2+\frac{h^2}{12}\langle R,w \rangle+\langle R,u\rangle,
    \end{equation}
    \begin{equation}\label{eq2.6}
     \langle w,u \rangle \leq-|u|_1^2-\frac{h^2}{18}\|w\|^2+\frac{h^2}{12}\langle R,w \rangle+\langle R,u \rangle.
    \end{equation}
   \end{lemma}

   \begin{lemma}\label{lemmma2.6}
    \cite{wangxp} For any $u,w,R$ defined on $\omega_{h\tau}$ satisfying 
    \begin{equation*}
       \begin{cases}
        w_i^n=\delta_x^2 u_i^n-\frac{h^2}{12}\delta_x^2 w_i^n+R_i^n,\quad 1\leq i \leq M-1,\quad 0\leq n\leq N,\\
        u_0^n=u_M^n=w_0^n=w_M^n=0,\quad 0\leq n\leq N,
      \end{cases}  
    \end{equation*}   
      we can obtain
      \begin{equation*}
        \begin{split}
          \langle w^{n-\frac{1}{2}},\delta_t u^{n-\frac{1}{2}} \rangle=&-\frac{1}{2\tau}\left[(|u^n|_1^2-|u^{n-1}|_1^2)+\frac{h^2}{12}(\|w^n\|^2-\|w^{n-1}\|^2)-\frac{h^4}{144}(|w^n|_1^2-|w^{n-1}|_1^2)\right]\\
          &+\frac{h^2}{12}\langle w^{n-\frac{1}{2}},\delta_t R^{n-\frac{1}{2}}\rangle+\langle R^{n-\frac{1}{2}},\delta_t u^{n-\frac{1}{2}}\rangle, \quad 1 \leq n \leq N.
        \end{split}
      \end{equation*}
   \end{lemma}
   \begin{remark}
    In this paper, the notation $\tilde{C}$ denotes a generic constant with different values in different situations, but it is independent of the spatial step $h$ and temporal step $\tau$.  
   \end{remark}
 \section{Derivation of the compact difference scheme}
 Let $w:=u_{xx}, v:=u_{t}$, then the problem (\ref*{eq1.1})-(\ref*{eq1.3}) is equivalent to
\begin{align}
    \label{eq3.1}  &\mu_1\mathbb{D}_{t}^{(\alpha+1)}u+\mu_2\mathbb{D}_{t}^{(\alpha)}u+uu_x=\lambda w(x,t), \quad  (x,t)\in (0,L)\times (0,T],\quad \alpha\in (0, 1),\\
    \label{eq3.2}  &w(x,t)=u_{xx}(x,t),\qquad (x,t)\in (0,L)\times (0,T],\\
    \label{eq3.3}  &u(x,0)=\varphi_1 (x),\quad v(x,0)=\varphi_2 (x), \qquad x \in [0,L],\\
    \label{eq3.4}  &u(0,t)=u(L,t)=0, \quad t \in [0, T].
\end{align}
  According to (\ref*{eq3.1}) and (\ref*{eq3.4}), we can easily get 
  \begin{equation}\label{eqb3.5}
     w(0,t)=w(L,t)=0, \quad t \in [0, T]. 
  \end{equation}
  Let $U:=\{U_i^n|0 \leq i \leq M,0 \leq n \leq N\}$ and $W:=\{W_i^n|0 \leq i \leq M,0 \leq n \leq N\}$ denote the grid functions defined on $\omega_{h\tau}$, where $U_i^n:=u(x_i,t_n)$ and $W_i^n:=w(x_i,t_n)$.
  \vskip 0.2mm
  Considering (\ref*{eq3.1}) at the point $(x_i,t_{n-\frac{1}{2}})$ and using Lemma \ref*{lemma2.2} and Lemma \ref*{lemma2.3}, we have 
  \begin{equation}\label{eqn3.5}
    \begin{split}
      &\frac{\tau^{-1}}{\Gamma(1-\alpha)}\mathcal{D}_t^{(\alpha)}( \mu_1 \delta_t U_i^{n-\frac{1}{2}}+\mu_2  U_i^{n-\frac{1}{2}})+\psi(U_i^{n-\frac{1}{2}},U_i^{n-\frac{1}{2}})-\frac{h^2}{2}\psi(W_i^{n-\frac{1}{2}},U_i^{n-\frac{1}{2}})\\
      &=\lambda W_i^{n-\frac{1}{2}}+P_i^{n-\frac{1}{2}},\quad 1\leq i \leq M-1,\quad 1\leq n\leq N,
    \end{split}
  \end{equation}
  from which we can get the following estimate for truncation error 
  \begin{equation}\label{eqn3.6}
    \left| P_i^{n-\frac{1}{2}}\right| \leq \tilde{C}(\tau^{2-\alpha}+h^4),\quad 1\leq i \leq M-1,\quad 1\leq n\leq N.
  \end{equation}
  \vskip 0.2mm
  Meanwhile, considering (\ref*{eq3.2}) at the point $(x_i,t_{n})$ and using Taylor expansion, it's easy to get
  \begin{equation}\label{eqn3.7}
    W_i^n=\delta_x^2 U_i^n-\frac{h^2}{12}\delta_x^2 W_i^n+Q_i^n,\quad 1\leq i \leq M-1,\quad 0\leq n\leq N,
  \end{equation}
from which $Q_i^n$ and $\delta_t Q_i^{n-\frac{1}{2}}$ are estimated by
\begin{align}
  \label{eqn3.8} \left| Q_i^{n}\right| \leq \tilde{C}h^4, \quad 1\leq i \leq M-1,\quad 0\leq n\leq N,\\
  \label{eqn3.9} \left| \delta_t Q_i^{n-\frac{1}{2}}\right| \leq \tilde{C}h^4, \quad 1\leq i \leq M-1,\quad 1\leq n\leq N.
\end{align}
\vskip 0.2mm
Omitting the truncation errors $P_i^{n-\frac{1}{2}}$ and $Q_i^n$ in (\ref*{eqn3.5}) and (\ref*{eqn3.7}), replacing the functions $U_i^n,V_i^n,W_i^n$ with their numerical approximation $u_i^n,v_i^n,w_i^n$, respectively, and combining with (\ref*{eq3.3})-(\ref*{eqb3.5}), we construct the following compact difference scheme
\begin{equation}\label{eqn3.10}
  \begin{split}
    &\frac{\tau^{-1}}{\Gamma(1-\alpha)}\mathcal{D}_t^{(\alpha)}( \mu_1 \delta_t u_i^{n-\frac{1}{2}}+\mu_2  u_i^{n-\frac{1}{2}})+\psi(u_i^{n-\frac{1}{2}},u_i^{n-\frac{1}{2}})-\frac{h^2}{2}\psi(w_i^{n-\frac{1}{2}},u_i^{n-\frac{1}{2}})\\
    &-\lambda w_i^{n-\frac{1}{2}}=0,\quad 1\leq i \leq M-1,\quad 1\leq n\leq N,
  \end{split}
\end{equation}
\begin{align}
  \label{eqn3.11} &w_i^n=\delta_x^2 u_i^n-\frac{h^2}{12}\delta_x^2 w_i^n,\quad 1\leq i \leq M-1,\quad 0\leq n\leq N,\\
  \label{eqn3.12} &u_i^0=\varphi_1 (x_i),\quad v_i^0=\varphi_2 (x_i), \quad 0\leq i \leq M,\\
  \label{eqn3.13} &u_0^n=u_M^n=w_0^n=w_M^n=0,\quad 0\leq n\leq N.
\end{align}

\section{Analysis of convergence and stability}
   \vskip 0.2mm
   In this section, the convergence and stability of the compact difference scheme (\ref*{eqn3.10})-(\ref*{eqn3.13}) will be derived. First, we introduce some lemmas as follows.
   \begin{lemma}\label{lem4.2}
    \cite{qsx3,sunzz} For any $\tilde{G}=\{g_1,g_2,g_3,\cdots\}$ and $\hat{q}$, we have the following estimate
    \begin{equation*}
      \begin{split}
        &\sum\limits_{n=1}^{N}\left[ b_0g_n-\sum\limits_{i=1}^{n-1}(b_{n-i-1}-b_{n-i})g_i-b_{n-1}\hat{q}\right]g_n\\
        & \geq \frac{T^{-\alpha}}{2}\tau\sum\limits_{n=1}^{N}(g_n)^2-\frac{T^{1-\alpha}}{2(1-\alpha)}\hat{q}^2, \qquad 0<\alpha <1, \quad N=1,2,3,\cdots
      \end{split}
    \end{equation*}
   where $b_i(i \geq 0)$ defined in (\ref*{eq2.2}).
    \end{lemma}
    \begin{lemma}\label{lem4.3}
      For any grid functions $u \in \Omega_h$ and $v \in \mathring{\Omega}_h$, then we have 
      \begin{equation}\label{eqn4.2}
        \langle \psi(u^{n-\frac{1}{2}},v^{n-\frac{1}{2}}),\delta_t v^{n-\frac{1}{2}}\rangle=\langle \psi(u^{n-\frac{1}{2}},v^{n-1}),\delta_t v^{n-\frac{1}{2}}\rangle, \quad 1 \leq n \leq N.
      \end{equation}
     \end{lemma}
     \begin{proof}
      Using Lemma $\ref*{lemma2.4}$, we can arrive at
      \begin{equation*}
        \begin{split}
          &\left\langle \psi(u^{n-\frac{1}{2}},v^{n-\frac{1}{2}}),\delta_t v^{n-\frac{1}{2}}\right\rangle\\
          =&\left\langle \psi(u^{n-\frac{1}{2}},v^{n-1}+\frac{\tau}{2}\delta_t v^{n-\frac{1}{2}}),\delta_t v^{n-\frac{1}{2}}\right\rangle\\
          =&\left\langle \psi(u^{n-\frac{1}{2}},v^{n-1}),\delta_t v^{n-\frac{1}{2}}\right\rangle+\frac{\tau}{2}\left\langle \psi(u^{n-\frac{1}{2}},\delta_t v^{n-\frac{1}{2}}),\delta_t v^{n-\frac{1}{2}}\right\rangle\\
          =&\left\langle \psi(u^{n-\frac{1}{2}},v^{n-1}),\delta_t v^{n-\frac{1}{2}}\right\rangle.
        \end{split}
      \end{equation*}
     \end{proof}
    \begin{lemma}\label{lem4.4}
      \cite{gi} (Discrete Gr$\ddot{o}$nwall's inequality) If $a_n$ is a non-negative real sequence and satisfies 
      $$a_n\leq b_n+\sum\limits_{i=0}^{n-1}d_i a_i,\quad n\geq 1,$$
      where $b_n$ is a non-descending and non-negative sequence, $d_n \geq 0$, then we arrive at
      $$a_n\leq b_n\exp \left( \sum\limits_{i=0}^{n-1}d_i\right) ,\quad n\geq 1.$$
    \end{lemma}
    \subsection{Convergence}
    The convergence of compact difference scheme (\ref*{eqn3.10})-(\ref*{eqn3.13}) will be analysed in the following. Firstly we give some error notations and constants as follows
    \begin{equation*}
      \begin{array}{lll}
        e_i^n:=U_i^n-u_i^n, \quad \sigma_i^n:=V_i^n-v_i^n,\quad \rho_i^n :=W_i^n-w_i^n,\\
        \\
        c_0:=\max \limits_{(x,t)\in [0,L]\times[0,T]}\left\{ |u(x,t)|,|u_{x}(x,t)|,|u_{xx}(x,t)|,|u_{xxx}(x,t)|\right\},\\
        \\
        c_1:=\frac{T^{-\alpha}}{2\Gamma(1-\alpha)}, \quad c_2:=\frac{1}{c_1\mu_1^2}, \quad c_3:=c_2\mu_1(c_0+\frac{c_0L}{\sqrt{6}}+\frac{\sqrt{L}}{2}), \quad c_4:=c_2\mu_2c_0(L^2+L),\\
        \\
        c_5:=c_2\mu_1c_0(2+L), \quad c_6:=c_2\mu_1(2c_0+\sqrt{L}), \quad c_7:=c_2\mu_2c_0L,\\
        \\
        c_8:=\max\left\{ \frac{2c_3^2 +c_4+3c_5^2+\frac{5\mu_2^2L^2}{6\mu_1^2}+L^2+1}{c_2\mu_1\lambda} ,\quad \frac{15c_6^2+3c_7^2+6}{c_2\mu_1\lambda} +\frac{2}{3} ,\quad 1 \right\}  ,\\
        \\
        c_9:=\left( \frac{2c_2\mu_1}{\lambda}+\frac{c_2\mu_2^2}{2\mu_1\lambda}+\frac{c_2\lambda_2}{\mu_1}(2\mu_1^2+\mu_2^2 )+c_2\mu_1\lambda \right)\tilde{C}^2L,\quad c_{10}:=e^{6c_8T}\sqrt{3\tilde{C}+6Tc_9}.
      \end{array}
    \end{equation*}
    \vskip 0.2mm
    Substracting (\ref*{eqn3.10}) and (\ref*{eqn3.11}) from (\ref*{eqn3.5}) and (\ref*{eqn3.7}) respectively, then we can get error equations as follows
    \begin{equation}\label{eqn4.3}
      \begin{split}
        &\frac{\tau^{-1}}{\Gamma(1-\alpha)}\mathcal{D}_t^{(\alpha)}( \mu_1 \delta_t e_i^{n-\frac{1}{2}}+\mu_2  e_i^{n-\frac{1}{2}})+\left[\psi(U_i^{n-\frac{1}{2}},U_i^{n-\frac{1}{2}})-\psi(u_i^{n-\frac{1}{2}},u_i^{n-\frac{1}{2}})\right]\\
        &-\frac{h^2}{2}\left[\psi(W_i^{n-\frac{1}{2}},U_i^{n-\frac{1}{2}})-\psi(w_i^{n-\frac{1}{2}},u_i^{n-\frac{1}{2}})\right]-\lambda \rho_i^{n-\frac{1}{2}}=P_i^{n-\frac{1}{2}},\\
        &\qquad \qquad 1\leq i \leq M-1,\quad 1\leq n\leq N,
      \end{split}
    \end{equation}
    \begin{align}
      \label{eqn4,4} &\rho_i^n=\delta_x^2 e_i^n-\frac{h^2}{12}\delta_x^2 \rho_i^n+Q_i^n,\quad 1\leq i \leq M-1,\quad 0\leq n\leq N,\\
      \label{eqn4,5} &e_i^0=0,\quad \sigma_i^0=0, \quad 0\leq i \leq M,\\
      \label{eqn4,6} &e_0^n=e_M^n=\rho_0^n=\rho_M^n=0,\quad 0\leq n\leq N.
    \end{align}
    \begin{theorem}\label{Th4.1}
      Assume that the problem (\ref*{eq3.1})-(\ref*{eq3.4}) has  solutions $u(x,t),w(x,t)$, and $\{u_i^n,w_i^n|0\leq i \leq M, 0\leq n\leq N\}$ are the solutions of the compact difference scheme (\ref*{eqn3.10})-(\ref*{eqn3.13}). If $\tau$ and $h$ satisfy $\tau^{2-\alpha} +h^4 \leq 1/c_{10}$ and $c_8\tau \leq 1/3$, then we have
      \begin{equation}\label{eqn4.7}
        |e^n|_1 \leq c_{10}(\tau^{2-\alpha}+h^4),\quad 0\leq n\leq N.
      \end{equation}
    \end{theorem}
    \begin{proof}
     we adopt mathematical induction to prove this theorem.
      \vskip 0.2mm
      Step 1: It is easy to know (\ref*{eqn4.7}) holds for $n=0$. Taking $n=0$ in (\ref*{eqn4,4}) and noticing (\ref*{eqn4,5}) and (\ref*{eqn4,6}), we have
      \begin{equation}\label{eqn4.8}
        \rho_i^0=-\frac{h^2}{12}\delta_x^2 \rho_i^0+Q_i^0,\quad 1\leq i \leq M-1.
      \end{equation}
      Taking an inner product of (\ref*{eqn4.8}) with $\rho^0$ and using (\ref*{eq2.1}), we have 
      \begin{equation*}
        \|\rho^0\|^2=\frac{h^2}{12}|\rho^0|_1^2+\langle Q^0, \rho^0 \rangle \leq \frac{1}{3}\|\rho^0\|^2+\frac{1}{3}\|\rho^0\|^2+\frac{3}{4}\|Q^0\|^2,
      \end{equation*}
      combining with (\ref*{eqn3.8}), we can get
      \begin{equation}\label{eqn4.9}
        \|\rho^0\|^2 \leq \frac{9}{4}\|Q^0\|^2 \leq \frac{9}{4}L(\tilde{C}h^4)^2\leq\tilde{C}(\tau^{2-\alpha}+h^4)^2.
      \end{equation}
      \vskip 0.2mm
      Step 2: Assume that (\ref*{eqn4.7}) holds for $0\leq n\leq N-1$. When $\tau^{2-\alpha} +h^4 \leq 1/c_{10}$, according to (\ref*{eq2.1}), we can get
      \begin{equation}\label{eqn4.10}
        |e^n|_1 \leq 1, \quad \|e^n\| \leq \frac{L}{\sqrt{6}}|e^n|_1 \leq\frac{L}{\sqrt{6}}, \quad \|e^n\|_{\infty}\leq \frac{\sqrt{L}}{2}|e^n|_1\leq \frac{\sqrt{L}}{2},\quad 0\leq n\leq N-1.
      \end{equation}
      \vskip 0.2mm
      Step 3: Next, we need to prove that (\ref*{eqn4.7}) holds for $n=N$. Taking an inner product of (\ref*{eqn4.3}) with $(\mu_1 \delta_t e^{n-\frac{1}{2}}+\mu_2  e^{n-\frac{1}{2}})$ and summing for $n$ from $1$ to $N$, we denote
      \begin{equation*}
        \begin{split}
          &\Lambda_1^{n-\frac{1}{2}}:=\sum\limits_{n=1}^{N}\left\langle \frac{\tau^{-1}}{\Gamma(1-\alpha)}\mathcal{D}_t^{(\alpha)}( \mu_1 \delta_t e^{n-\frac{1}{2}}+\mu_2  e^{n-\frac{1}{2}}),\mu_1 \delta_t e^{n-\frac{1}{2}}+\mu_2  e^{n-\frac{1}{2}}\right\rangle,\\
          &\Lambda_2^{n-\frac{1}{2}}:=-\left\langle \psi(U^{n-\frac{1}{2}},U^{n-\frac{1}{2}})-\psi(u^{n-\frac{1}{2}},u^{n-\frac{1}{2}}),\mu_1 \delta_t e^{n-\frac{1}{2}}+\mu_2  e^{n-\frac{1}{2}} \right\rangle,\\
          &\Lambda_3^{n-\frac{1}{2}}:=\frac{h^2}{2}\left\langle \psi(W^{n-\frac{1}{2}},U^{n-\frac{1}{2}})-\psi(w^{n-\frac{1}{2}},u^{n-\frac{1}{2}}),\mu_1 \delta_t e^{n-\frac{1}{2}}+\mu_2  e^{n-\frac{1}{2}} \right\rangle,\\
          &\Lambda_4^{n-\frac{1}{2}}:=\lambda\left\langle \rho^{n-\frac{1}{2}},\mu_1 \delta_t e^{n-\frac{1}{2}}+\mu_2  e^{n-\frac{1}{2}} \right\rangle,\\
          &\Lambda_5^{n-\frac{1}{2}}:=\left\langle P^{n-\frac{1}{2}},\mu_1 \delta_t e^{n-\frac{1}{2}}+\mu_2  e^{n-\frac{1}{2}} \right\rangle.\\
        \end{split}
      \end{equation*}
      Therefore, we get the following equation
      \begin{equation}\label{eqn4.11}
        \Lambda_1^{n-\frac{1}{2}}-\sum\limits_{n=1}^{N}\Lambda_2^{n-\frac{1}{2}}-\sum\limits_{n=1}^{N}\Lambda_3^{n-\frac{1}{2}}-\sum\limits_{n=1}^{N}\Lambda_4^{n-\frac{1}{2}}=\sum\limits_{n=1}^{N}\Lambda_5^{n-\frac{1}{2}}.
      \end{equation}
      Considering $\Lambda_1^{n-\frac{1}{2}}$, using  Lemma \ref*{lem4.2}, and noticing (\ref*{eqn4,5}), we have
      \begin{equation}\label{eqn4.12}
        \begin{split}
          \Lambda_1^{n-\frac{1}{2}} &\geq \sum\limits_{n=1}^{N} \frac{T^{-\alpha}}{2\Gamma(1-\alpha)}\| \mu_1 \delta_t e^{n-\frac{1}{2}}+\mu_2  e^{n-\frac{1}{2}}\|^2-\frac{T^{1-\alpha}}{2\tau \Gamma(2-\alpha)}\| \mu_1 \sigma^0+\mu_2  e^0\|^2\\
          &=\sum\limits_{n=1}^{N} \frac{T^{-\alpha}}{2\Gamma(1-\alpha)}\| \mu_1 \delta_t e^{n-\frac{1}{2}}+\mu_2  e^{n-\frac{1}{2}}\|^2.\\
        \end{split}
      \end{equation}
      Thus, (\ref*{eqn4.11}) turns into 
      \begin{equation}\label{eqn4.13}
        \sum\limits_{n=1}^{N}c_1\| \mu_1 \delta_t e^{n-\frac{1}{2}}+\mu_2  e^{n-\frac{1}{2}}\|^2-\sum\limits_{n=1}^{N}\Lambda_4^{n-\frac{1}{2}} \leq \sum\limits_{n=1}^{N}\Lambda_2^{n-\frac{1}{2}}+\sum\limits_{n=1}^{N}\Lambda_3^{n-\frac{1}{2}}+\sum\limits_{n=1}^{N}\Lambda_5^{n-\frac{1}{2}}.
      \end{equation}
      Due to 
      \begin{equation*}
          \| \mu_1 \delta_t e^{n-\frac{1}{2}}+\mu_2 e^{n-\frac{1}{2}}\|^2 \geq \left(\| \mu_1 \delta_t e^{n-\frac{1}{2}}\|-\|\mu_2 e^{n-\frac{1}{2}} \|\right)^2
          \geq  \mu_1^2\| \delta_t e^{n-\frac{1}{2}}\|^2-2\mu_1\mu_2\| \delta_t e^{n-\frac{1}{2}}\| \| e^{n-\frac{1}{2}} \|,
      \end{equation*}
      then, (\ref*{eqn4.13}) becomes 
      \begin{equation*}
        \sum\limits_{n=1}^{N}\left( c_1\mu_1^2\|\delta_t e^{n-\frac{1}{2}}\|^2-\Lambda_4^{n-\frac{1}{2}}\right)  \leq \sum\limits_{n=1}^{N} \left(2c_1\mu_1\mu_2\|e^{n-\frac{1}{2}}\|\| \delta_te^{n-\frac{1}{2}}\|+\Lambda_2^{n-\frac{1}{2}}+\Lambda_3^{n-\frac{1}{2}}+\Lambda_5^{n-\frac{1}{2}}\right).
      \end{equation*}
      Furthermore, we have
      \begin{equation*}
        \sum\limits_{n=1}^{N}\left( \|\delta_t e^{n-\frac{1}{2}}\|^2-\frac{1}{c_1\mu_1^2}\Lambda_4^{n-\frac{1}{2}}\right)  \leq \sum\limits_{n=1}^{N} \left(\frac{2\mu_2}{\mu_1}\|e^{n-\frac{1}{2}}\|\| \delta_te^{n-\frac{1}{2}}\|+\frac{1}{c_1\mu_1^2}(\Lambda_2^{n-\frac{1}{2}}+\Lambda_3^{n-\frac{1}{2}}+\Lambda_5^{n-\frac{1}{2}})\right),
      \end{equation*}
      that is 
      \begin{equation}\label{eqn4.14}
        \sum\limits_{n=1}^{N}\left( \|\delta_t e^{n-\frac{1}{2}}\|^2-c_2\Lambda_4^{n-\frac{1}{2}}\right)  \leq \sum\limits_{n=1}^{N} \left(\frac{2\mu_2}{\mu_1}\|e^{n-\frac{1}{2}}\|\| \delta_te^{n-\frac{1}{2}}\|+c_2\Lambda_2^{n-\frac{1}{2}}+c_2\Lambda_3^{n-\frac{1}{2}}+c_2\Lambda_5^{n-\frac{1}{2}}\right).
      \end{equation}
      \vskip 0.2mm
      Next, we analyze $c_2\Lambda_2^{n-\frac{1}{2}},c_2\Lambda_3^{n-\frac{1}{2}},c_2\Lambda_4^{n-\frac{1}{2}},c_2\Lambda_5^{n-\frac{1}{2}} $ and $\frac{2\mu_2}{\mu_1}\|e^{n-\frac{1}{2}}\|\|\delta_t e^{n-\frac{1}{2}}\|$ one by one.
      \vskip 0.2mm
      \textbf{(\Rmnum{1}).} Considering $c_2\Lambda_2^{n-\frac{1}{2}}$, according to $u_i^n=U_i^n-e_i^n$ and bilinearity of $\psi$, we have
      \begin{equation*}
        \begin{split}
          &\psi(U^{n-\frac{1}{2}},U^{n-\frac{1}{2}})-\psi(u^{n-\frac{1}{2}},u^{n-\frac{1}{2}})\\
          =&\psi(U^{n-\frac{1}{2}},e^{n-\frac{1}{2}})+\psi(e^{n-\frac{1}{2}},U^{n-\frac{1}{2}})-\psi(e^{n-\frac{1}{2}},e^{n-\frac{1}{2}}),
        \end{split}
      \end{equation*}
      then
      \[c_2\Lambda_2^{n-\frac{1}{2}}=-c_2\left\langle \psi(U^{n-\frac{1}{2}},e^{n-\frac{1}{2}})+\psi(e^{n-\frac{1}{2}},U^{n-\frac{1}{2}})-\psi(e^{n-\frac{1}{2}},e^{n-\frac{1}{2}}),\mu_1 \delta_t e^{n-\frac{1}{2}}+\mu_2 e^{n-\frac{1}{2}} \right\rangle,\]
      and noticing Lemma \ref*{lemma2.4} and Lemma \ref*{lem4.3}, we have
      \begin{equation*}
        \begin{split}
          c_2\Lambda_2^{n-\frac{1}{2}}=&-c_2\mu_1\left\langle \psi(U^{n-\frac{1}{2}},e^{n-\frac{1}{2}})+\psi(e^{n-\frac{1}{2}},U^{n-\frac{1}{2}}),\delta_t e^{n-\frac{1}{2}} \right\rangle\\
          &+c_2\mu_1\left\langle \psi(e^{n-\frac{1}{2}},e^{n-1}),\delta_t e^{n-\frac{1}{2}} \right\rangle-c_2\mu_2\left\langle \psi(e^{n-\frac{1}{2}},U^{n-\frac{1}{2}}),e^{n-\frac{1}{2}} \right\rangle\\
          =&-c_2\mu_1\frac{h}{3}\sum\limits_{i=1}^{M-1}\left[ U_i^{n-\frac{1}{2}} \Delta_x e_i^{n-\frac{1}{2}}+2\Delta_x(U_i^{n-\frac{1}{2}}e_i^{n-\frac{1}{2}}) +e_i^{n-\frac{1}{2}}\Delta_x U_i^{n-\frac{1}{2}} \right]\delta_t e_i^{n-\frac{1}{2}}\\
          &+c_2\mu_1\frac{h}{3}\sum\limits_{i=1}^{M-1}\left[ e_i^{n-\frac{1}{2}} \Delta_x e_i^{n-1}+\Delta_x(e_i^{n-\frac{1}{2}}e_i^{n-1}) \right]\delta_t e_i^{n-\frac{1}{2}}\\
          &-c_2\mu_2\frac{h}{3}\sum\limits_{i=1}^{M-1}\left[ e_i^{n-\frac{1}{2}} \Delta_x U_i^{n-\frac{1}{2}}+\Delta_x(e_i^{n-\frac{1}{2}}U_i^{n-\frac{1}{2}}) \right]e_i^{n-\frac{1}{2}}\\
          =&-c_2\mu_1\frac{h}{3}\sum\limits_{i=1}^{M-1}\left[ 3U_i^{n-\frac{1}{2}} \Delta_x e_i^{n-\frac{1}{2}}+e_{i+1}^{n-\frac{1}{2}}\delta_x U_{i+\frac{1}{2}}^{n-\frac{1}{2}}+e_{i-1}^{n-\frac{1}{2}}\delta_x U_{i-\frac{1}{2}}^{n-\frac{1}{2}} +e_i^{n-\frac{1}{2}}\Delta_x U_i^{n-\frac{1}{2}} \right]\delta_t e_i^{n-\frac{1}{2}}\\
          &+c_2\mu_1\frac{h}{3}\sum\limits_{i=1}^{M-1}\left[ 2e_i^{n-\frac{1}{2}} \Delta_x e_i^{n-1}+\frac{1}{2}e_{i+1}^{n-1}\delta_x e_{i+\frac{1}{2}}^{n-\frac{1}{2}}+\frac{1}{2}e_{i-1}^{n-1}\delta_x e_{i-\frac{1}{2}}^{n-\frac{1}{2}} \right]\delta_t e_i^{n-\frac{1}{2}}\\
          &-c_2\mu_2\frac{h}{3}\sum\limits_{i=1}^{M-1}\left[ 2e_i^{n-\frac{1}{2}} \Delta_x U_i^{n-\frac{1}{2}}+\frac{1}{2}U_{i+1}^{n-\frac{1}{2}}\delta_x e_{i+\frac{1}{2}}^{n-\frac{1}{2}}+\frac{1}{2}U_{i-1}^{n-\frac{1}{2}}\delta_x e_{i-\frac{1}{2}}^{n-\frac{1}{2}} \right]e_i^{n-\frac{1}{2}}.\\
        \end{split}
      \end{equation*}
      Using (\ref*{eqn4.10}), Cauchy-Schwarz inequality, Young inequality, and Lemma \ref*{lemma2.1}, we can get
      \begin{equation*}
        \begin{split}
          c_2\Lambda_2^{n-\frac{1}{2}} \leq &c_2\mu_1c_0|e^{n-\frac{1}{2}}|_1\|\delta_t e^{n-\frac{1}{2}}\|+c_2\mu_1c_0\|e^{n-\frac{1}{2}}\|\|\delta_t e^{n-\frac{1}{2}}\|  \\
          &+\frac{c_2\mu_1}{3}\left[ 2\|e^{n-\frac{1}{2}}\|_{\infty}|e^{n-1}|_1+\|e^{n-1}\|_{\infty}|e^{n-\frac{1}{2}}|_1 \right]\|\delta_t e^{n-\frac{1}{2}}\|  \\
           &+\frac{c_2\mu_2c_0}{3}\left[ 2\|e^{n-\frac{1}{2}}\|^2+|e^{n-\frac{1}{2}}|_1\|e^{n-\frac{1}{2}}\| \right]\\
           \leq &c_2\mu_1c_0|e^{n-\frac{1}{2}}|_1\|\delta_t e^{n-\frac{1}{2}}\|+c_2\mu_1c_0\frac{L}{\sqrt{6}}|e^{n-\frac{1}{2}}|_1\|\delta_t e^{n-\frac{1}{2}}\|  \\
           &+\frac{c_2\mu_1}{3}\left[ 2\frac{\sqrt{L}}{2}|e^{n-\frac{1}{2}}|_1 \cdot 1+\frac{\sqrt{L}}{2}|e^{n-\frac{1}{2}}|_1 \right]\|\delta_t e^{n-\frac{1}{2}}\|  \\
            &+\frac{c_2\mu_2c_0}{3}\left[ 2\frac{L^2}{6} |e^{n-\frac{1}{2}}|_1^2+\frac{L}{\sqrt{6}}|e^{n-\frac{1}{2}}|_1^2 \right]\\
          \leq &c_2\mu_1(c_0+\frac{c_0L}{\sqrt{6}}+\frac{\sqrt{L}}{2})|e^{n-\frac{1}{2}}|_1\|\delta_t e^{n-\frac{1}{2}}\|+c_2\mu_2c_0(L^2+L)|e^{n-\frac{1}{2}}|_1^2\\
          =&c_3|e^{n-\frac{1}{2}}|_1\|\delta_t e^{n-\frac{1}{2}}\|+c_4|e^{n-\frac{1}{2}}|_1^2\\
          \leq & \frac{1}{5}\|\delta_t e^{n-\frac{1}{2}}\|^2+(\frac{5c_3^2}{4}+c_4)|e^{n-\frac{1}{2}}|_1^2.
        \end{split}
      \end{equation*}
      \vskip 0.2mm
      \textbf{(\Rmnum{2}).} Similarly, considering $c_2\Lambda_3^{n-\frac{1}{2}}$, according to $u_i^n=U_i^n-e_i^n$ and $w_i^n=W_i^n-\rho_i^n$, we have
      \begin{equation*}
        \begin{split}
          &\psi(W^{n-\frac{1}{2}},U^{n-\frac{1}{2}})-\psi(w^{n-\frac{1}{2}},u^{n-\frac{1}{2}})\\
          =&\psi(\rho^{n-\frac{1}{2}},U^{n-\frac{1}{2}})+\psi(W^{n-\frac{1}{2}},e^{n-\frac{1}{2}})-\psi(\rho^{n-\frac{1}{2}},e^{n-\frac{1}{2}}),
        \end{split}
      \end{equation*}
      then, using Lemma \ref*{lemma2.4} and Lemma \ref*{lem4.3}, we can get
      \begin{equation*}
      \begin{split}
        c_2\Lambda_3^{n-\frac{1}{2}}=&c_2\mu_1\frac{h^2}{2}\left\langle \psi(W^{n-\frac{1}{2}},e^{n-\frac{1}{2}})+\psi(\rho^{n-\frac{1}{2}},U^{n-\frac{1}{2}})-\psi(\rho^{n-\frac{1}{2}},e^{n-1}),\delta_t e^{n-\frac{1}{2}} \right\rangle\\
        &+c_2\mu_2\frac{h^2}{2}\left\langle \psi(\rho^{n-\frac{1}{2}},U^{n-\frac{1}{2}}),e^{n-\frac{1}{2}} \right\rangle\\
        =&c_2\mu_1\frac{h^3}{6}\sum\limits_{i=1}^{M-1}\left[ 2W_i^{n-\frac{1}{2}} \Delta_x e_i^{n-\frac{1}{2}}+\frac{1}{2}e_{i+1}^{n-\frac{1}{2}}\delta_x W_{i+\frac{1}{2}}^{n-\frac{1}{2}}+\frac{1}{2}e_{i-1}^{n-\frac{1}{2}}\delta_x W_{i-\frac{1}{2}}^{n-\frac{1}{2}} \right]\delta_t e_i^{n-\frac{1}{2}}\\
        &+c_2\mu_1\frac{h^3}{6}\sum\limits_{i=1}^{M-1}\left[ 2\rho_i^{n-\frac{1}{2}} \Delta_x U_i^{n-\frac{1}{2}}+\frac{1}{2}U_{i+1}^{n-\frac{1}{2}}\delta_x \rho_{i+\frac{1}{2}}^{n-\frac{1}{2}}+\frac{1}{2}U_{i-1}^{n-\frac{1}{2}}\delta_x \rho_{i-\frac{1}{2}}^{n-\frac{1}{2}} \right]\delta_t e_i^{n-\frac{1}{2}}\\
        &-c_2\mu_1\frac{h^3}{6}\sum\limits_{i=1}^{M-1}\left[ 2\rho_i^{n-\frac{1}{2}} \Delta_x e_i^{n-1}+\frac{1}{2}e_{i+1}^{n-1}\delta_x \rho_{i+\frac{1}{2}}^{n-\frac{1}{2}}+\frac{1}{2}e_{i-1}^{n-1}\delta_x \rho_{i-\frac{1}{2}}^{n-\frac{1}{2}} \right]\delta_t e_i^{n-\frac{1}{2}}\\
        &+c_2\mu_2\frac{h^3}{6}\sum\limits_{i=1}^{M-1}\left[ 2\rho_i^{n-\frac{1}{2}} \Delta_x U_i^{n-\frac{1}{2}}+\frac{1}{2}U_{i+1}^{n-\frac{1}{2}}\delta_x \rho_{i+\frac{1}{2}}^{n-\frac{1}{2}}+\frac{1}{2}U_{i-1}^{n-\frac{1}{2}}\delta_x \rho_{i-\frac{1}{2}}^{n-\frac{1}{2}} \right] e_i^{n-\frac{1}{2}}.\\
      \end{split}
    \end{equation*}
    Combining with (\ref*{eqn4.10}) and Lemma \ref*{lemma2.1}, and using Cauchy-Schwarz inequality and Young inequality, we can get
    \begin{equation*}
      \begin{split}
        c_2\Lambda_3^{n-\frac{1}{2}} \leq &c_2\mu_1\frac{h^2}{6}\left[ 2c_0|e^{n-\frac{1}{2}}|_1+c_0\|e^{n-\frac{1}{2}}\|+2c_0\|\rho^{n-\frac{1}{2}}\|+c_0|\rho^{n-\frac{1}{2}}|_1 \right]\|\delta_t e^{n-\frac{1}{2}}\|\\
        &+c_2\mu_1\frac{h^2}{6}\left[ 2\|\rho^{n-\frac{1}{2}}\|_{\infty}|e^{n-1}|_1+\|e^{n-1}\|_{\infty}|\rho^{n-\frac{1}{2}}|_1 \right]\|\delta_t e^{n-\frac{1}{2}}\|\\
        &+c_2\mu_2\frac{h^2}{6}\left[ 2c_0\|\rho^{n-\frac{1}{2}}\|+c_0|\rho^{n-\frac{1}{2}}|_1 \right]\| e^{n-\frac{1}{2}}\|,
      \end{split}
    \end{equation*}
    then we have
     \begin{equation*}
      \begin{split}
        c_2\Lambda_3^{n-\frac{1}{2}} \leq &c_2\mu_1\frac{h^2}{6}(2c_0+\frac{c_0L}{\sqrt{6}})|e^{n-\frac{1}{2}}|_1\|\delta_t e^{n-\frac{1}{2}}\|+c_2\mu_1\frac{h^2}{3}(c_0+\frac{c_0}{h})\|\rho^{n-\frac{1}{2}}\|\|\delta_t e^{n-\frac{1}{2}}\|\\
        &+c_2\mu_1h\frac{\sqrt{L}}{2}\|\rho^{n-\frac{1}{2}}\|\|\delta_t e^{n-\frac{1}{2}}\|+c_2\mu_2c_0\frac{h^2}{3}(1+\frac{1}{h})\frac{L}{\sqrt{6}}\|\rho^{n-\frac{1}{2}}\||e^{n-\frac{1}{2}}|_1\\
        \leq &c_2\mu_1c_0(2+L)|e^{n-\frac{1}{2}}|_1\|\delta_t e^{n-\frac{1}{2}}\|+c_2\mu_1(2c_0+\sqrt{L})h\|\rho^{n-\frac{1}{2}}\|\|\delta_t e^{n-\frac{1}{2}}\| \\
        &+c_2\mu_2c_0Lh\|\rho^{n-\frac{1}{2}}\||e^{n-\frac{1}{2}}|_1\\
        =&c_5|e^{n-\frac{1}{2}}|_1\|\delta_t e^{n-\frac{1}{2}}\|+c_6h\|\rho^{n-\frac{1}{2}}\|\|\delta_t e^{n-\frac{1}{2}}\|+c_7h\|\rho^{n-\frac{1}{2}}\||e^{n-\frac{1}{2}}|_1\\
        \leq &\frac{1}{10}\| \delta_t e^{n-\frac{1}{2}}\|^2+\frac{5}{2}c_5^2|e^{n-\frac{1}{2}} |_1^2+\frac{1}{10}\| \delta_t e^{n-\frac{1}{2}}\|^2+\frac{5}{2}c_6^2h^2\| \rho^{n-\frac{1}{2}}\|^2+\frac{1}{2}|e^{n-\frac{1}{2}} |_1^2+\frac{1}{2}c_7^2h^2\| \rho^{n-\frac{1}{2}}\|^2 \\
        \leq &\frac{1}{5}\|\delta_t e^{n-\frac{1}{2}}\|^2+(\frac{5}{2}c_5^2+\frac{1}{2})|e^{n-\frac{1}{2}}|_1^2+(\frac{5}{2}c_6^2+\frac{1}{2}c_7^2)h^2\|\rho^{n-\frac{1}{2}}\|^2.
      \end{split}
    \end{equation*}
    \vskip 0.2mm
    \textbf{(\Rmnum{3}).} Considering $c_2\Lambda_4^{n-\frac{1}{2}}$. Firstly,
     \begin{equation*}
        c_2\Lambda_4^{n-\frac{1}{2}}=c_2\mu_1\lambda\left\langle \rho^{n-\frac{1}{2}},\delta_t e^{n-\frac{1}{2}} \right\rangle+c_2\mu_2\lambda\left\langle \rho^{n-\frac{1}{2}},e^{n-\frac{1}{2}} \right\rangle.
    \end{equation*}
     Then using Lemma \ref*{lemma2.5} and Lemma \ref*{lemmma2.6}, we can get
    \begin{equation*}
      \begin{split}
        -c_2\Lambda_4^{n-\frac{1}{2}} \geq &\frac{c_2\mu_1\lambda}{2\tau}\left[(|e^n|_1^2-|e^{n-1}|_1^2)+\frac{h^2}{12}(\|\rho^n\|^2-\|\rho^{n-1}\|^2)-\frac{h^4}{144}(|\rho^n|_1^2-|\rho^{n-1}|_1^2)\right]\\
        &-\frac{c_2\mu_1\lambda h^2}{12}\langle \rho^{n-\frac{1}{2}},\delta_t Q^{n-\frac{1}{2}}\rangle-c_2\mu_1\lambda\langle Q^{n-\frac{1}{2}},\delta_t e^{n-\frac{1}{2}}\rangle \\
        &+c_2\mu_2\lambda\left[|e^{n-\frac{1}{2}}|_1^2+\frac{h^2}{18}\|\rho^{n-\frac{1}{2}}\|_1^2\right]-\frac{c_2\mu_2\lambda h^2}{12}\langle  \rho^{n-\frac{1}{2}},Q^{n-\frac{1}{2}} \rangle-c_2\mu_2\lambda\langle Q^{n-\frac{1}{2}},e^{n-\frac{1}{2}} \rangle\\
        \geq &\frac{c_2\mu_1\lambda}{2\tau}\left[(|e^n|_1^2-|e^{n-1}|_1^2)+\frac{h^2}{12}(\|\rho^n\|^2-\|\rho^{n-1}\|^2)-\frac{h^4}{144}(|\rho^n|_1^2-|\rho^{n-1}|_1^2)\right]\\
        &-\frac{c_2\mu_1\lambda h^2}{12}\langle \rho^{n-\frac{1}{2}},\delta_t Q^{n-\frac{1}{2}}\rangle-c_2\mu_1\lambda\langle Q^{n-\frac{1}{2}},\delta_t e^{n-\frac{1}{2}}\rangle \\
        &-\frac{c_2\mu_2\lambda h^2}{12}\langle  \rho^{n-\frac{1}{2}},Q^{n-\frac{1}{2}} \rangle-c_2\mu_2\lambda\langle Q^{n-\frac{1}{2}},e^{n-\frac{1}{2}} \rangle\\
        \geq & \frac{c_2\mu_1\lambda}{2\tau}\left[(|e^n|_1^2-|e^{n-1}|_1^2)+\frac{h^2}{12}(\|\rho^n\|^2-\|\rho^{n-1}\|^2)-\frac{h^4}{144}(|\rho^n|_1^2-|\rho^{n-1}|_1^2)\right]\\
        &-\left( c_2\mu_1\lambda h\| \rho^{n-\frac{1}{2}}\| \|\delta_t Q^{n-\frac{1}{2}} \|+c_2\mu_1\lambda\| Q^{n-\frac{1}{2}}\|\|\delta_t e^{n-\frac{1}{2}} \| \right)\\
        &-\left( c_2\mu_2\lambda h\| \rho^{n-\frac{1}{2}}\| \|Q^{n-\frac{1}{2}} \|+c_2\mu_2\lambda\| Q^{n-\frac{1}{2}}\|\|e^{n-\frac{1}{2}} \| \right)\\
        \geq &\frac{c_2\mu_1\lambda}{2\tau}\left[(|e^n|_1^2-|e^{n-1}|_1^2)+\frac{h^2}{12}(\|\rho^n\|^2-\|\rho^{n-1}\|^2)-\frac{h^4}{144}(|\rho^n|_1^2-|\rho^{n-1}|_1^2)\right]\\
        &-\Big( \frac{1}{5}\|\delta_t e^{n-\frac{1}{2}}\|^2+h^2\| \rho^{n-\frac{1}{2}}\|^2+\frac{L^2}{12}|e^{n-\frac{1}{2}}|_1^2+\frac{1}{2}(c_2\mu_1\lambda)^2\|\delta_t Q^{n-\frac{1}{2}} \|^2 \\
         &  + c_2^2 \lambda^2(\frac{5}{4}\mu_1^2+\mu_2^2)\|Q^{n-\frac{1}{2}} \|^2\Big). \\
      \end{split}
    \end{equation*}
    \vskip 0.2mm
    \textbf{(\Rmnum{4}).} Considering $c_2\Lambda_5^{n-\frac{1}{2}}$, we have
    \begin{equation*}
      \begin{split}
        c_2\Lambda_5^{n-\frac{1}{2}}\leq &c_2\mu_1\|P^{n-\frac{1}{2}}\|\|\delta_t e^{n-\frac{1}{2}}\|+c_2\mu_2\|P^{n-\frac{1}{2}}\|\|e^{n-\frac{1}{2}}\|\\
        \leq &\frac{1}{5}\|\delta_t e^{n-\frac{1}{2}}\|^2+(\frac{5}{4}c_2^2\mu_1^2+\frac{1}{2}c_2^2\mu_2^2)\|P^{n-\frac{1}{2}}\|^2 +\frac{L^2}{12}|e^{n-\frac{1}{2}}|_1^2.
      \end{split}
    \end{equation*}
    \vskip 0.2mm
    \textbf{(\Rmnum{5}).} Finally, we have
    \begin{equation*}
      \frac{2\mu_2}{\mu_1}\|e^{n-\frac{1}{2}}\|\| \delta_te^{n-\frac{1}{2}}\| \leq \frac{1}{5}\| \delta_te^{n-\frac{1}{2}}\|+\frac{5\mu_2^2}{\mu_1^2}\|e^{n-\frac{1}{2}}\|^2 \leq \frac{1}{5}\| \delta_te^{n-\frac{1}{2}}\|+\frac{5\mu_2^2L^2}{6\mu_1^2}|e^{n-\frac{1}{2}}|_1^2.
    \end{equation*}
    \vskip 0.2mm
    Now we substitute the analysis results of \textbf{(\Rmnum{1})}-\textbf{(\Rmnum{5})} back into (\ref*{eqn4.14}), we have
   \begin{equation*}
      \begin{split}
& \sum\limits_{n=1}^{N}\biggl\{ \|\delta_t e^{n-\frac{1}{2}}\|^2+\frac{c_2\mu_1\lambda}{2\tau}\left[(|e^n|_1^2-|e^{n-1}|_1^2)+\frac{h^2}{12}(\|\rho^n\|^2-\|\rho^{n-1}\|^2)-\frac{h^4}{144}(|\rho^n|_1^2-|\rho^{n-1}|_1^2)\right] \biggr\}\\
\leq & \sum\limits_{n=1}^{N} \biggl\{\|\delta_t e^{n-\frac{1}{2}}\|^2+( \frac{5}{4}c_3^2 +c_4+\frac{5}{2}c_5^2+\frac{L^2}{6}+\frac{5\mu_2^2L^2}{6\mu_1^2}+\frac{1}{2} )|e^{n-\frac{1}{2}}|_1^2 +( 15c_6^2+3c_7^2+6 )\frac{h^2}{6}\| \rho^{n-\frac{1}{2}}\|^2 \\
&+( \frac{5}{4}c_2^2\mu_1^2+\frac{1}{2}c_2^2\mu_2^2 )\|P^{n-\frac{1}{2}}\|^2+c_2^2\lambda_2^2(\frac{5}{4}\mu_1^2+\mu_2^2  )\|Q^{n-\frac{1}{2}} \|^2+\frac{1}{2}(c_2\mu_1\lambda )^2\|\delta_t Q^{n-\frac{1}{2}} \|^2 \biggr\}.
      \end{split}
    \end{equation*}
    Furthermore, we have
\begin{equation*}
  \begin{split}
    & \sum\limits_{n=1}^{N}\biggl\{ \frac{1}{2\tau}\left[(|e^n|_1^2-|e^{n-1}|_1^2)+\frac{h^2}{12}(\|\rho^n\|^2-\|\rho^{n-1}\|^2)-\frac{h^4}{144}(|\rho^n|_1^2-|\rho^{n-1}|_1^2)\right]\biggr\}\\
    \leq & \sum\limits_{n=1}^{N}\biggl\{ \Big( \frac{2c_3^2 +c_4+3c_5^2+\frac{5\mu_2^2L^2}{6\mu_1^2}+L^2+1}{c_2\mu_1\lambda} \Big)|e^{n-\frac{1}{2}}|_1^2+\Big( \frac{15c_6^2+3c_7^2+6}{c_2\mu_1\lambda} \Big)\frac{h^2}{6}\| \rho^{n-\frac{1}{2}}\|^2\\
    &+\Big( \frac{2c_2\mu_1}{\lambda}+\frac{c_2\mu_2^2}{2\mu_1\lambda} \Big)\|P^{n-\frac{1}{2}}\|^2+\frac{c_2\lambda_2}{\mu_1}\left(2\mu_1^2+\mu_2^2  \right)\|Q^{n-\frac{1}{2}} \|^2+c_2\mu_1\lambda \|\delta_t Q^{n-\frac{1}{2}} \|^2\biggr\}\\
    \leq & \sum\limits_{n=1}^{N}\biggl\{ \Big( \frac{2c_3^2 +c_4+3c_5^2+\frac{5\mu_2^2L^2}{6\mu_1^2}+L^2+1}{c_2\mu_1\lambda} \Big)|e^{n-\frac{1}{2}}|_1^2+\Big( \frac{15c_6^2+3c_7^2+6}{c_2\mu_1\lambda} +\frac{2}{3}\Big)\frac{h^2}{6}\| \rho^{n-\frac{1}{2}}\|^2-\frac{h^4}{36}| \rho^{n-\frac{1}{2}}|_1^2\\
    &+\Big( \frac{2c_2\mu_1}{\lambda}+\frac{c_2\mu_2^2}{2\mu_1\lambda} \Big)\|P^{n-\frac{1}{2}}\|^2+\frac{c_2\lambda_2}{\mu_1}\left(2\mu_1^2+\mu_2^2  \right)\|Q^{n-\frac{1}{2}} \|^2+c_2\mu_1\lambda \|\delta_t Q^{n-\frac{1}{2}} \|^2 \biggr\}. \\
  \end{split}
\end{equation*}
Define
    \begin{equation*}
      H^n:=|e^n|_1^2+\frac{h^2}{12}\|\rho^n\|^2-\frac{h^4}{144}|\rho^n|_1^2, \quad 0 \leq n\leq N.
    \end{equation*}
    It is easy to konw that $0 \leq \frac{h^2}{18}\|\rho^0\|^2 \leq H^0 \leq \frac{h^2}{12}\|\rho^0\|^2 $ by the inverse estimate.
    Combining (\ref*{eqn3.6})-(\ref*{eqn3.9}), we can get
    \begin{equation*}
      \sum\limits_{n=1}^{N}\Bigl[\frac{1}{2\tau}(H^n-H^{n-1})\Bigr]\leq \sum\limits_{n=1}^{N}\Bigl[c_8(H^n+H^{n-1})+c_9(\tau^{2-\alpha}+h^4)^2\Bigr].
    \end{equation*}
    Then
    \begin{equation*}
        \frac{1}{2\tau}(H^{N}-H^0) \leq c_8H^{N}+2c_8\sum\limits_{n=0}^{N-1} H^n+\sum\limits_{n=0}^{N-1} c_9(\tau^{2-\alpha}+h^4)^2,
    \end{equation*}
  and when $\tau c_8 \leq 1/3$, we can get
  \begin{equation*}
    H^{N}\leq 3H^0+\sum\limits_{n=0}^{N-1} 6\tau c_9(\tau^{2-\alpha}+h^4)^2+\sum\limits_{n=0}^{N-1} 12\tau c_8 H^n, \quad 0 \leq n\leq N-1.
\end{equation*}
    Combining Lemma \ref*{lem4.4} and (\ref*{eqn4.9}), we can get
    \begin{equation*}
      \begin{split}
        H^{N}\leq e^{(12c_8T)}\left(3H^0+6Tc_9(\tau^{2-\alpha}+h^4)^2\right) \leq c_{10}^2(\tau^{2-\alpha}+h^4)^2,
      \end{split}
    \end{equation*}
    and it is easy to know 
    \begin{equation*}
      H^{N} \geq |e^{N}|_1^2+\frac{h^2}{18}\|\rho^{N}\|^2 \geq |e^{N}|_1^2.
    \end{equation*}
    Finally, we can obtain
    \begin{equation*}
      |e^{N}|_1\leq c_{10}(\tau^{2-\alpha}+h^4).
    \end{equation*}
    This completes the proof.
    \end{proof}
    \begin{remark}
      Under the conditions of Theorem \ref*{Th4.1}, combining with (\ref*{eq2.1}), we can obtain
      \begin{equation}
        \|e^n\| \leq \frac{c_{10}L}{\sqrt{6}}(\tau^{2-\alpha}+h^4),\quad \|e^n\|_{\infty}\leq \frac{c_{10}\sqrt{L}}{2}(\tau^{2-\alpha}+h^4),\quad 0 \leq n \leq N.
      \end{equation}
    \end{remark}
\subsection{Stability}
    \vskip 0.2mm
    In the following, we consider the stability of the compact difference scheme (\ref*{eqn3.10})-(\ref*{eqn3.13}). Now supposing some perturbation terms are added into (\ref*{eqn3.10})-(\ref*{eqn3.13}), we get the following system of equations
    \begin{equation}\label{eqn4.16}
      \begin{split}
        &\frac{\tau^{-1}}{\Gamma(1-\alpha)}\mathcal{D}_t^{(\alpha)}( \mu_1 \delta_t \tilde{u}_i^{n-\frac{1}{2}}+\mu_2  \tilde{u}_i^{n-\frac{1}{2}})+\psi(\tilde{u}_i^{n-\frac{1}{2}},\tilde{u}_i^{n-\frac{1}{2}})-\frac{h^2}{2}\psi(\tilde{w}_i^{n-\frac{1}{2}},\tilde{u}_i^{n-\frac{1}{2}})\\
        &-\lambda \tilde{w}_i^{n-\frac{1}{2}}=f_i^{n-\frac{1}{2}},\quad 1\leq i \leq M-1,\quad 1\leq n\leq N,
      \end{split}
    \end{equation}
    \begin{align}
      \label{eqn4.17} &\tilde{w}_i^n=\delta_x^2 \tilde{u}_i^n-\frac{h^2}{12}\delta_x^2 \tilde{w}_i^n,\quad 1\leq i \leq M-1,\quad 0\leq n\leq N,\\
      \label{eqn4.18} &\tilde{u}_i^0=\varphi_1 (x_i)+r_1(x_i),\quad \tilde{v}_i^0=\varphi_2 (x_i)+r_2(x_i), \quad 0\leq i \leq M,\\
      \label{eqn4.19} &\tilde{u}_0^n=\tilde{u}_M^n=\tilde{w}_0^n=\tilde{w}_M^n=0,\quad 0\leq n\leq N.
    \end{align}
    Denote
    \begin{equation*}
      \beta_i^n=\tilde{u}_i^n-u_i^n,\quad \gamma_i^n=\tilde{w}_i^n
      -w_i^n,\quad \eta_i^0=\tilde{v}_i^0-v_i^0,\quad 0\leq i \leq M,\quad 0\leq n\leq N.
    \end{equation*}
    Substracting (\ref*{eqn3.10})-(\ref*{eqn3.13}) from (\ref*{eqn4.16})-(\ref*{eqn4.19}), respectively, we get the residual equations as follows
    \begin{equation}\label{eqn4.20}
      \begin{split}
        &\frac{\tau^{-1}}{\Gamma(1-\alpha)}\mathcal{D}_t^{(\alpha)}( \mu_1 \delta_t \beta_i^{n-\frac{1}{2}}+\mu_2  \beta_i^{n-\frac{1}{2}})+\left[\psi(\tilde{u}_i^{n-\frac{1}{2}},\tilde{u}_i^{n-\frac{1}{2}})-\psi(u_i^{n-\frac{1}{2}},u_i^{n-\frac{1}{2}})\right]\\
        &-\frac{h^2}{2}\left[\psi(\tilde{w}_i^{n-\frac{1}{2}},\tilde{u}_i^{n-\frac{1}{2}})-\psi(w_i^{n-\frac{1}{2}},u_i^{n-\frac{1}{2}})\right]-\lambda \gamma_i^{n-\frac{1}{2}}=f_i^{n-\frac{1}{2}},\\
        &\qquad \qquad 1\leq i \leq M-1,\quad 1\leq n\leq N,
      \end{split}
    \end{equation}
    \begin{align}
      \label{eqn4,21} &\gamma_i^n=\delta_x^2 \beta_i^n-\frac{h^2}{12}\delta_x^2 \gamma_i^n,\quad 1\leq i \leq M-1,\quad 0\leq n\leq N,\\
      \label{eqn4,22} &\beta_i^0=r_1(x_i),\quad \eta_i^0=r_2(x_i), \quad 0\leq i \leq M,\\
      \label{eqn4,23} &\beta_0^n=\beta_M^n=\gamma_0^n=\gamma_M^n=0,\quad 0\leq n\leq N.
    \end{align}
    \vskip 0.2mm
    For purpose of facilitating the following description of stability, we denote some coefficients as follows
    \begin{equation*}
      \begin{split}
        &c_{11}:=\frac{T^{1-\alpha}}{2\Gamma(2-\alpha)},\quad c_{12}:=\frac{c_{11}}{c_1\mu_1^2},\quad c_{13}:=\frac{c_{12}}{c_2\mu_1\lambda},\\
        & \hat{c}_3:=c_2\mu_1(c_0+\frac{c_0L}{\sqrt{6}}+\frac{\sqrt{L}}{2}C^{\star}), \quad \hat{c}_6:=c_2\mu_1(2c_0+\sqrt{L}C^{\star}),\\
      \end{split}
    \end{equation*}
    \begin{equation*}
      \begin{split}
        &c_{14}:=\max\left\{ \frac{2\hat{c}_3^2 +c_4+2c_5^2+\frac{5\mu_2^2L^2}{6\mu_1^2}+L^2+1}{c_2\mu_1\lambda} ,\quad \frac{8\hat{c}_6^2+3c_7^2}{c_2\mu_1\lambda} +\frac{2}{3} ,\quad 1 \right\},\\
        &c_{15}:=\frac{2c_{2}\mu_1}{\lambda}+\frac{c_{2}\mu_2^2}{2\mu_1\lambda}, \\
      \end{split}
    \end{equation*}
    from which, $C^{\star}$ will be given in the following theorem.
    \begin{theorem}\label{Th4.2}
      Suppose $\{ \beta_i^n,\gamma_i^n \}$ is the solution of (\ref*{eqn4.20})-(\ref*{eqn4,23}), for any $0\leq n\leq N$, when $\tau c_{14} \leq 1/3$, the following inequality holds
      \begin{equation}\label{eqn4.24}
        |\beta^{n}|_1\leq e^{(6c_{14}T)}\left(\frac{21}{4}|\beta^0|_1^2+6c_{13}\| \mu_1 \eta^0+\mu_2\beta^0\|^2+6\tau \sum\limits_{n=0}^{N-1}c_{15}\|f^{n+\frac{1}{2}}\|^2 \right)^{1/2}.
      \end{equation}
    \end{theorem}
    \begin{proof}
      It is obvious that (\ref*{eqn4.24}) holds for $n=0$.
      \vskip 0.2mm
      Assume (\ref*{eqn4.24}) is vaild for $1\leq n\leq N-1$. $\exists C^{\star} >0$, such that $|\beta^n|_1 \leq C^{\star}$ $(1\leq n\leq N-1)$, combining with Lemma \ref*{lemma2.1}, we can arrive at
      \begin{equation}\label{eqn4.25}
       \|\beta^n\|\leq\frac{L}{\sqrt{6}}C^{\star}, \quad \|\beta^n\|_{\infty}\leq \frac{\sqrt{L}}{2}C^{\star},\quad 1\leq n\leq N-1.
      \end{equation}
      Below we need to prove that (\ref*{eqn4.24}) is valid for $n=N$. Similar to the proof of convergence, taking the inner product of (\ref*{eqn4.24}) with $(\mu_1 \delta_t \beta^{n-\frac{1}{2}}+\mu_2  \beta^{n-\frac{1}{2}})$ and summing for $n$ from $1$ to $N$, then we have
      \begin{equation}\label{eqn4.26}
        \Pi_1^{n-\frac{1}{2}}-\sum\limits_{n=1}^{N}\Pi_2^{n-\frac{1}{2}}-\sum\limits_{n=1}^{N}\Pi_3^{n-\frac{1}{2}}-\sum\limits_{n=1}^{N}\Pi_4^{n-\frac{1}{2}}=\sum\limits_{n=1}^{N}\Pi_5^{n-\frac{1}{2}},
      \end{equation}
      where
      \begin{equation*}
        \begin{split}
          &\Pi_1^{n-\frac{1}{2}}:=\sum\limits_{n=1}^{N}\left\langle \frac{\tau^{-1}}{\Gamma(1-\alpha)}\mathcal{D}_t^{(\alpha)}( \mu_1 \delta_t \beta^{n-\frac{1}{2}}+\mu_2  \beta^{n-\frac{1}{2}}),\mu_1 \delta_t \beta^{n-\frac{1}{2}}+\mu_2  \beta^{n-\frac{1}{2}}\right\rangle,\\
          &\Pi_2^{n-\frac{1}{2}}:=-\left\langle \psi(\tilde{u}^{n-\frac{1}{2}},\tilde{u}^{n-\frac{1}{2}})-\psi(u^{n-\frac{1}{2}},u^{n-\frac{1}{2}}),\mu_1 \delta_t \beta^{n-\frac{1}{2}}+\mu_2  \beta^{n-\frac{1}{2}} \right\rangle,\\
          &\Pi_3^{n-\frac{1}{2}}:=\frac{h^2}{2}\left\langle \psi(\tilde{w}^{n-\frac{1}{2}},\tilde{u}^{n-\frac{1}{2}})-\psi(w^{n-\frac{1}{2}},u^{n-\frac{1}{2}}),\mu_1 \delta_t \beta^{n-\frac{1}{2}}+\mu_2  \beta^{n-\frac{1}{2}} \right\rangle,\\
          &\Pi_4^{n-\frac{1}{2}}:=\lambda\left\langle \gamma^{n-\frac{1}{2}},\mu_1 \delta_t \beta^{n-\frac{1}{2}}+\mu_2  \beta^{n-\frac{1}{2}} \right\rangle,\\
          &\Pi_5^{n-\frac{1}{2}}:=\left\langle f^{n-\frac{1}{2}},\mu_1 \delta_t \beta^{n-\frac{1}{2}}+\mu_2  \beta^{n-\frac{1}{2}} \right\rangle.\\
        \end{split}
      \end{equation*}
      First, considering $\Pi_1^{n-\frac{1}{2}}$, combining with Lemma \ref*{lem4.2} and (\ref*{eqn4,22}), we have 
      \begin{equation*}
        \begin{split}
          \Pi_1^{n-\frac{1}{2}} &\geq \sum\limits_{n=1}^{N} \frac{T^{-\alpha}}{2\Gamma(1-\alpha)}\| \mu_1 \delta_t \beta^{n-\frac{1}{2}}+\mu_2  \beta^{n-\frac{1}{2}}\|^2-\frac{T^{1-\alpha}}{2\tau \Gamma(2-\alpha)}\| \mu_1 \eta^0+\mu_2  \beta^0\|^2.\\
        \end{split}
      \end{equation*}
      Therefore, we have
      \begin{equation*}
        \begin{split}
          &\sum\limits_{n=1}^{N} c_1\| \mu_1 \delta_t \beta^{n-\frac{1}{2}}+\mu_2  \beta^{n-\frac{1}{2}}\|^2-c_{11}\tau^{-1}\| \mu_1 \eta^0+\mu_2  \beta^0\|^2-\sum\limits_{n=1}^{N}\Pi_4^{n-\frac{1}{2}}\\
          &\leq \sum\limits_{n=1}^{N}\Pi_2^{n-\frac{1}{2}}+\sum\limits_{n=1}^{N}\Pi_3^{n-\frac{1}{2}}+\sum\limits_{n=1}^{N}\Pi_5^{n-\frac{1}{2}},\\
        \end{split}
      \end{equation*}
      furthermore, the above inequality is transformed into
      \begin{equation}\label{eqn4.27}
        \begin{split}
          \sum\limits_{n=1}^{N} \biggl\{ \|\delta_t \beta^{n-\frac{1}{2}}\|^2-c_2\Pi_4^{n-\frac{1}{2}} \biggr\}\leq& \sum\limits_{n=1}^{N}\biggl\{ c_2\Pi_2^{n-\frac{1}{2}}+c_2\Pi_3^{n-\frac{1}{2}}+c_2\Pi_5^{n-\frac{1}{2}}+\frac{2\mu_2}{\mu_1}\| \delta_t\beta^{n-\frac{1}{2}}\|\| \beta^{n-\frac{1}{2}}\|\biggr\} \\
          &+c_{12}\tau^{-1}\| \mu_1 \eta^0+\mu_2  \beta^0\|^2.\\
        \end{split}
      \end{equation}
    Similarly, we analyze $c_2\Pi_2^{n-\frac{1}{2}},c_2\Pi_3^{n-\frac{1}{2}},c_2\Pi_4^{n-\frac{1}{2}},c_2\Pi_5^{n-\frac{1}{2}} $ and $\frac{2\mu_2}{\mu_1}\|\beta^{n-\frac{1}{2}}\|\| \delta_t\beta^{n-\frac{1}{2}}\|$, respectively. Since the process of analysis is the same as \textbf{(\Rmnum{1})}-\textbf{(\Rmnum{4})}, we omit the detail of analysis process and here give their corresponding estimates directly
    \begin{align}
      \label{eqn4.28} &c_2\Pi_2^{n-\frac{1}{2}} \leq \frac{1}{5}\|\delta_t \beta^{n-\frac{1}{2}}\|^2+(\frac{5\hat{c}_3^2}{4}+c_4)|\beta^{n-\frac{1}{2}}|_1^2,\\
      \label{eqn4.29} &c_2\Pi_3^{n-\frac{1}{2}} \leq \frac{2}{5}\|\delta_t \beta^{n-\frac{1}{2}}\|^2+(\frac{5}{4}c_5^2+\frac{1}{2})|\beta^{n-\frac{1}{2}}|_1^2+(\frac{5}{4}\hat{c}_6^2+\frac{1}{2}c_7^2)h^2\|\gamma^{n-\frac{1}{2}}\|^2,\\
      \label{eqn4.30} -&c_2\Pi_4^{n-\frac{1}{2}} \geq \frac{c_2\mu_1\lambda}{2\tau}\left[(|\beta^n|_1^2-|\beta^{n-1}|_1^2)+\frac{h^2}{12}(\|\gamma^n\|^2-\|\gamma^{n-1}\|^2)-\frac{h^4}{144}(|\gamma^n|_1^2-|\gamma^{n-1}|_1^2)\right],\\
      \label{eqn4.31}&c_2\Pi_5^{n-\frac{1}{2}} \leq \frac{1}{5}\|\delta_t \beta^{n-\frac{1}{2}}\|^2+(\frac{5}{4}c_2^2\mu_1^2+\frac{1}{2}c_2^2\mu_2^2)\|f^{n-\frac{1}{2}}\|^2 +\frac{L^2}{12}|\beta^{n-\frac{1}{2}}|_1^2,\\
      \label{eqn4.32} & \frac{2\mu_2}{\mu_1}\|\beta^{n-\frac{1}{2}}\|\| \delta_t\beta^{n-\frac{1}{2}}\| \leq \frac{1}{5}\| \delta_t\beta^{n-\frac{1}{2}}\|+\frac{5\mu_2^2L^2}{6\mu_1^2}|\beta^{n-\frac{1}{2}}|_1^2.
    \end{align} 
    \vskip 0.2mm
    Now substituting (\ref*{eqn4.28})-(\ref*{eqn4.32}) back into (\ref*{eqn4.27}), we can get
    \begin{equation*}
      \begin{split}
        &\sum\limits_{n=1}^{N}\biggl\{\|\delta_t \beta^{n-\frac{1}{2}}\|^2+\frac{c_2\mu_1\lambda}{2\tau}\left[(|\beta^n|_1^2-|\beta^{n-1}|_1^2)+\frac{h^2}{12}(\|\gamma^n\|^2-\|\gamma^{n-1}\|^2)-\frac{h^4}{144}(|\gamma^n|_1^2-|\gamma^{n-1}|_1^2)\right]\biggr\}\\
        \leq &\sum\limits_{n=1}^{N}\biggl\{ \|\delta_t \beta^{n-\frac{1}{2}}\|^2+\left( \frac{5}{4}\hat{c}_3^2 +c_4+\frac{5}{4}c_5^2+\frac{L^2}{12}+\frac{5\mu_2^2L^2}{6\mu_1^2}+\frac{1}{2} \right)|\beta^{n-\frac{1}{2}}|_1^2+\left( \frac{15}{2}\hat{c}_6^2+3c_7^2 \right)\frac{h^2}{6}\| \gamma^{n-\frac{1}{2}}\|^2\\
        &+\left( \frac{5}{4}c_2^2\mu_1^2+\frac{1}{2}c_2^2\mu_2^2 \right)\|f^{n-\frac{1}{2}}\|^2\biggr\}+c_{12}\tau^{-1}\| \mu_1 \eta^0+\mu_2  \beta^0\|^2.\\
      \end{split}
    \end{equation*}
    Moreover, we have
    \begin{equation*}
      \begin{split}
        &\sum\limits_{n=1}^{N}\frac{1}{2\tau}\left[(|\beta^n|_1^2-|\beta^{n-1}|_1^2)+\frac{h^2}{12}(\|\gamma^n\|^2-\|\gamma^{n-1}\|^2)-\frac{h^4}{144}(|\gamma^n|_1^2-|\gamma^{n-1}|_1^2)\right]\\
        \leq &\sum\limits_{n=1}^{N}\biggl\{ \Big( \frac{2\hat{c}_3^2 +c_4+2c_5^2+\frac{5\mu_2^2L^2}{6\mu_1^2}+L^2+1}{c_2\mu_1\lambda} \Big)|\beta^{n-\frac{1}{2}}|_1^2+\Big( \frac{8\hat{c}_6^2+3c_7^2}{c_2\mu_1\lambda} \Big)\frac{h^2}{6}\| \gamma^{n-\frac{1}{2}}\|^2\biggr\}\\
        &+\sum\limits_{n=1}^{N} \Big( \frac{2c_2\mu_1}{\lambda}+\frac{c_2\mu_2^2}{2\mu_1\lambda} \Big)\|f^{n-\frac{1}{2}}\|^2+\frac{c_{12}}{c_2\mu_1 \lambda}\tau^{-1}\| \mu_1 \eta^0+\mu_2  \beta^0\|^2\\
        \leq &\sum\limits_{n=1}^{N}\biggl\{\Big( \frac{2\hat{c}_3^2 +c_4+2c_5^2+\frac{5\mu_2^2L^2}{6\mu_1^2}+L^2+1}{c_2\mu_1\lambda} \Big)|\beta^{n-\frac{1}{2}}|_1^2+\left( \frac{8\hat{c}_6^2+3c_7^2}{c_2\mu_1\lambda} +\frac{2}{3}\right)\frac{h^2}{6}\| \gamma^{n-\frac{1}{2}}\|^2-\frac{h^4}{36}| \gamma^{n-\frac{1}{2}}|_1^2\biggr\}\\
        &+\sum\limits_{n=1}^{N} c_{15}\|f^{n-\frac{1}{2}}\|^2+c_{13}\tau^{-1}\| \mu_1 \eta^0+\mu_2  \beta^0\|^2 \\
      \end{split}
    \end{equation*}
    \begin{equation*}
      \begin{split}
        \leq& \sum\limits_{n=1}^{N}c_{14}\left[ (|\beta^n|_1^2+|\beta^{n-1}|_1^2)+\frac{h^2}{12}(\|\gamma^n\|^2+\|\gamma^{n-1}\|^2)-\frac{h^4}{144}(|\gamma^n|_1^2+|\gamma^{n-1}|_1^2) \right]\\
        &+\sum\limits_{n=1}^{N}c_{15}\|f^{n-\frac{1}{2}}\|^2+c_{13}\tau^{-1}\| \mu_1 \eta^0+\mu_2  \beta^0\|^2.\\
      \end{split}
    \end{equation*}
    \vskip 0.2mm
    Define
    \begin{equation*}
      \tilde{H}^n:=|\beta^n|_1^2+\frac{h^2}{12}\|\gamma^n\|^2-\frac{h^4}{144}|\gamma^n|_1^2, \quad 0 \leq n\leq N.
    \end{equation*}
    It is easy to konw $\tilde{H}^0 \geq |\beta^0|_1^2+\frac{h^2}{18}\|\gamma^0\|^2>0$, we can get
    \begin{equation*}
      \begin{array}{ccc}
       \frac{1}{2\tau}(\tilde{H}^N-\tilde{H}^0) \leq c_{14}\tilde{H}^{N}+2c_{14}\sum\limits_{n=0}^{N-1} \tilde{H}^n+\sum\limits_{n=1}^{N} c_{15}\|f^{n-\frac{1}{2}}\|^2+ c_{13}\tau^{-1}\| \mu_1 \eta^0+\mu_2  \beta^0\|^2,
      \end{array}
    \end{equation*}
    when $\tau c_{14} \leq 1/3$, then
    \begin{equation*}
      \begin{array}{ccc}
       \tilde{H}^N \leq 3\tilde{H}^0+ 6 c_{13}\| \mu_1 \eta^0+\mu_2  \beta^0\|^2+6\tau\sum\limits_{n=0}^{N-1}c_{15}\|f^{n+\frac{1}{2}}\|^2+\sum\limits_{n=0}^{N-1} 12\tau c_{14}\tilde{H}^n.
      \end{array}
    \end{equation*}
    Combining with Lemma \ref*{lem4.4}, we can get
    \begin{equation*}
      \begin{split}
        \tilde{H}^{N}\leq e^{(12c_{14}T)}\left(3\tilde{H}^0+6c_{13}\| \mu_1 \eta^0+\mu_2 \beta^0\|^2+6\tau \sum\limits_{n=0}^{N-1}c_{15}\|f^{n+\frac{1}{2}}\|^2 \right).
      \end{split}
    \end{equation*}
    Next, taking the inner product of \eqref{eqn4,21} (case $n=0$) with $\gamma^0$ and utilizing the inverse estimate, then we get
    $$\|\gamma^0 \|^2 \leq \frac{2}{h}|\beta^0|_1\| \gamma^0\|+\frac{1}{3}\|\gamma^0 \|^2,$$
    thus
    \begin{equation*}
      \begin{split}
         \|\gamma^0\| \leq \frac{3}{h} | \beta^0|_1,
      \end{split}
    \end{equation*} 
    and it is easy to yield 
    \begin{equation*}
      \tilde{H}^{N} \geq |\beta^{N}|_1^2+\frac{h^2}{18}\|\gamma^{N}\|^2 \geq |\beta^{N}|_1^2.
    \end{equation*}
    Then we can get
    \begin{equation*}
      |\beta^{N}|_1\leq e^{(6c_{14}T)}\left(\frac{21}{4}|\beta^0|_1^2+6c_{13}\| \mu_1 \eta^0+\mu_2 \beta^0\|^2+6\tau\sum\limits_{n=0}^{N-1}c_{15}\|f^{n+\frac{1}{2}}\|^2  \right)^{1/2}.
    \end{equation*}
    This completes the proof.
    \end{proof}
    \begin{remark}
      Under the conditions of Theorem \ref*{Th4.2}, combining with (\ref*{eq2.1}), for $0 \leq n \leq N$, we can obtain
      \begin{equation*}
       \begin{split}
        &\|\beta^{n}\|\leq \frac{L}{\sqrt{6}}e^{(6c_{14}T)}\left(\frac{21}{4}|\beta^0|_1^2+6c_{13}\| \mu_1 \eta^0+\mu_2 \beta^0\|^2+6\tau\sum\limits_{n=0}^{N-1}c_{15}\|f^{n+\frac{1}{2}}\|^2  \right)^{1/2}, \\
        &|\beta^{n}|_{\infty}\leq \frac{\sqrt{L}}{2}e^{(6c_{14}T)}\left(\frac{21}{4}|\beta^0|_1^2+6c_{13}\| \mu_1 \eta^0+\mu_2 \beta^0\|^2+6\tau\sum\limits_{n=0}^{N-1}c_{15}\|f^{n+\frac{1}{2}}\|^2  \right)^{1/2}.
       \end{split}
      \end{equation*}
    
    \end{remark}



        

  \section{Numerical experiment}

     \vskip 0.2mm
         In this section, all experiments will be performed by utilizing the software MATLAB R2021a on a MacOS 12.5 (64 bit) PC-Inter(R) Core(TM) i5 CPU 1.4 GHz and 8.0 GB of RAM. Three illustrative examples are given to demonstrate the efficiency and numerical accuracy of our scheme.

     \vskip 0.2mm
     Let $U_i^n$ and $u_i^n$ ($0 \leq i \leq M$, $0\leq n \leq N$) be the exact solutions and numerical solutions respectively, and we compute the problem (\ref*{eq1.1})-(\ref*{eq1.3}) by using the compact difference scheme (\ref*{eqn3.10})-(\ref*{eqn3.13}). In the implementation of the implicit scheme, we apply the method of fixed point iteration \cite{zhangqf}. 
     \vskip 0.2mm
     Denote
     \begin{equation*}
     \begin{array}{cc}
     E_{\infty}(\tau,h):= \max\limits_{ 0\leq n \leq N }\| U^n-u^n \|_{\infty}, \\
     Order^\tau:=\log_2\Big(\frac{E_{\infty}(2\tau,h)}{E_{\infty}(\tau,h)}\Big),\qquad
     Order^h:=\log_2\Big(\frac{E_{\infty}(\tau,2h)}{E_{\infty}(\tau,h)}\Big).
     \end{array}
     \end{equation*}
     \vskip 0.2mm
     In the following the fixed point iteration is applied to the implicit scheme (\ref*{eqn3.10})-(\ref*{eqn3.13}). It can be seen from (\ref*{eqn3.10})-(\ref*{eqn3.11}) that
     \begin{equation*}
      \left(I+\frac{h^2}{12}\delta_x^2 \right) \left( \frac{1}{\theta }\mathcal{D}_t^{(\alpha)}( \mu_1 \delta_t u_i^{n-\frac{1}{2}}+\mu_2  u_i^{n-\frac{1}{2}})+\psi(u_i^{n-\frac{1}{2}},u_i^{n-\frac{1}{2}})-\frac{h^2}{2}\psi(w_i^{n-\frac{1}{2}},u_i^{n-\frac{1}{2}})\right)
      =\lambda \delta_x^2 u_i^{n-\frac{1}{2}},
     \end{equation*}
     where $I$ denotes the identity operator, $\theta :=\tau^{\alpha}\Gamma(1-\alpha)$.
     Let $\xi _1:=\mu_1/\tau +\mu_2/2$ and $\xi_2:=\mu_2/2-\mu_1/\tau$, then we have 
     $$ \mu_1 \delta_t u_i^{n-\frac{1}{2}}+\mu_2  u_i^{n-\frac{1}{2}}=\xi_1u_i^n+\xi_2u_i^{n-1}.$$
     Next, we analyze the term $\psi(u_i^{n-\frac{1}{2}},u_i^{n-\frac{1}{2}}) $, that is
     $$\psi(u_i^{n-\frac{1}{2}},u_i^{n-\frac{1}{2}})=\frac{1}{4}\left( \psi(u_i^{n},u_i^{n})+\psi(u_i^{n},u_i^{n-1})+\psi(u_i^{n-1},u_i^{n})+\psi(u_i^{n-1},u_i^{n-1})\right). $$
     In addition, the expansion of $\psi(w_i^{n-\frac{1}{2}},u_i^{n-\frac{1}{2}})$ is similar to that of $\psi(u_i^{n-\frac{1}{2}},u_i^{n-\frac{1}{2}})$. Therefore, we can obtain
    \begin{equation*}
      \begin{cases}
          \left(I+\frac{h^2}{12}\delta_x^2 \right) \left( \mathcal{D}_t^{(\alpha)}( \xi_1 u_i^{n,k+1}+\xi_2 u_i^{n-1})+\theta \psi^{k+\frac{1}{2}}(u_i^{n-\frac{1}{2}},u_i^{n-\frac{1}{2}})-\frac{\theta h^2}{2}\psi(w_i^{n-\frac{1}{2},k},u_i^{n-\frac{1}{2},k})\right)\\
      =\frac{\theta\lambda}{2} \delta_x^2 (u_i^{n,k+1}+u_i^{n-1}),\\ 
      \\
      \psi^{k+\frac{1}{2}}(u_i^{n-\frac{1}{2}},u_i^{n-\frac{1}{2}})=\frac{1}{4}\left( \psi^{k+\frac{1}{2}}(u_i^{n},u_i^{n})+\psi(u_i^{n,k+1},u_i^{n-1})+\psi(u_i^{n-1},u_i^{n,k+1})+\psi(u_i^{n-1},u_i^{n-1})\right),\\
       \\
      \psi(w_i^{n-\frac{1}{2},k},u_i^{n-\frac{1}{2},k})=\frac{1}{4}\left( \psi(w_i^{n,k},u_i^{n,k})+\psi(w_i^{n,k},u_i^{n-1})+\psi(w_i^{n-1},u_i^{n,k})+\psi(w_i^{n-1},u_i^{n-1})\right),\\
       \\
      \psi^{k+\frac{1}{2}}(u_i^{n},u_i^{n})=\frac{u_{i-1}^{n,k}+u_{i}^{n,k}+u_{i+1}^{n,k}}{6h}(u_{i+1}^{n,k+1}-u_{i-1}^{n,k+1}),\\ 
       \\
      \left(I+\frac{h^2}{12}\delta_x^2 \right)w_i^{n,k}=\delta_x^2 u_i^{n,k},
    \end{cases}
  \end{equation*}
     where $1\leq i \leq M-1$, $1\leq n \leq N$, $k$ denotes the iteration times, the initial guess is given by $u_i^{n,0}=u_i^{n-1}$, and the stopping criterion is set as $\|u^{n,k+1}-u^{n,k} \|_{\infty} \leq 10^{-8}$.

     \vskip 0.2mm
     \textbf{Example 1.} In the first example, considering (\ref*{eq1.1})-(\ref*{eq1.3}) with $L=T=1$, we give the following exact solution
     $$u(x,t)=(t^{2+\alpha}+1)\sin(\pi x), $$
     then the initial conditions are $\psi_1(x)=\sin(\pi x)$ and $\psi_2(x)=0$, and the source term is
     	$$f(x,t)=\Gamma(3+\alpha)t\sin(\pi x)\left(\mu_1+\frac{\mu_2t}{2}\right)+\frac{\pi}{2}\sin(2\pi x)(t^{2+\alpha}+1)^2+\lambda\pi^2\sin(\pi x)(t^{2+\alpha}+1). $$

     \vskip 0.2mm
     Tables 1 and 2 present the $L^{\infty}$-norm errors and corresponding convergence orders for compact difference schemes (\ref*{eqn3.10})-(\ref*{eqn3.13}). As we can see, all numerical results are consistent with the theoretical results. In temporal and spatial directions, it can approximately arrive at order $(2-\alpha)$ and fourth order, respectively. Figure \ref{fig1} shows the temporal convergence order for different values $\alpha$ when $M=80$, $\mu_1=\mu_2=1$, and $\lambda=1$. The spatial convergence with fixed time step and coefficients ($\mu_1,\mu_2,\lambda$) is showed in Figure \ref{fig2}, when varying spatial step sizes for different $\alpha$.

 \begin{table} \label{tb1}
       \center
     	\caption{The $L^{\infty}$-norm errors and temporal convergence orders in viscous coefficients $\lambda=1$ and $0.01$, respectively.}
       \resizebox{\textwidth}{!}{    %
     	\begin{tabular}{cccccccccccccc} 
     		\toprule
         \multicolumn{14}{c}{Compact difference scheme (\ref*{eqn3.10})-(\ref*{eqn3.13}) with $M=80$, $\mu_1=\mu_2=1$}\\
        \midrule
        \multicolumn{7}{c}{$\lambda=1$} & \multicolumn{7}{c}{$\lambda=0.01$} \\
        \cmidrule{2-7}  \cmidrule{9-14}
          \multirow{2}{*}{\centering $N$} & \multicolumn{2}{c}{$\alpha=0.25$} & \multicolumn{2}{c}{$\alpha=0.5$} & \multicolumn{2}{c}{$\alpha=0.75$} & \multirow{2}{*}{\centering $N$}  & \multicolumn{2}{c}{$\alpha=0.25$} & \multicolumn{2}{c}{$\alpha=0.5$} & \multicolumn{2}{c}{$\alpha=0.75$} \\
        \cmidrule{2-3}  \cmidrule{4-5}  \cmidrule{6-7}  \cmidrule{9-10}  \cmidrule{11-12}  \cmidrule{13-14}
         & $E_{\infty}(\tau,h)$ & $Order^{\tau}$ & $E_{\infty}(\tau,h)$ & $Order^{\tau}$ & $E_{\infty}(\tau,h)$ & $Order^{\tau}$ &  &  $E_{\infty}(\tau,h)$ & $Order^{\tau}$ & $E_{\infty}(\tau,h)$ & $Order^{\tau}$ & $E_{\infty}(\tau,h)$ & $Order^{\tau}$   \\
         \midrule
         64   & 5.043e-05 & *    &2.917e-04 & *    &1.871e-03 &*     &  64  &5.219e-04  &*     &2.435e-03 &*    &7.983e-03 &* \\
         128  & 1.586e-05 &1.669 &1.667e-04 &1.478 &7.904e-04 &1.243 &  128 &1.687e-04  &1.630 &8.830e-04 &1.457 &3.384e-03 &1.234 \\   
     		 256  & 4.938e-06 &1.683 &3.813e-05 & 1.484&3.332e-04 &1.246 &  256 &5.297e-05  &1.677 &3.183e-04 &1.472 &1.430e-03 &1.243 \\	
         512  & 1.527e-06 &1.694 &1.357e-05 & 1.489&1.403e-04 &1.248 &  512 &1.597e-05  &1.730 &1.137e-04 &1.485 &6.030e-04 &1.246 \\    
         \bottomrule              
     	\end{tabular}}
   \end{table}

   \begin{table}\label{tb2}
    \center
    \caption{The $L^{\infty}$-norm errors and spatial convergence orders in viscous coefficients $\lambda=1$ and $0.01$, respectively.}
    \resizebox{\textwidth}{!}{
    \begin{tabular}{cccccccccccccc} 
      \toprule
      \multicolumn{14}{c}{Compact difference scheme (\ref*{eqn3.10})-(\ref*{eqn3.13}) with $N=20000$, $\mu_1=\mu_2=1$}\\
     \midrule
     \multicolumn{7}{c}{$\lambda=1$} & \multicolumn{7}{c}{$\lambda=0.01$} \\
     \cmidrule{2-7}  \cmidrule{9-14}
       \multirow{2}{*}{\centering $M$} & \multicolumn{2}{c}{$\alpha=0.25$} & \multicolumn{2}{c}{$\alpha=0.5$} & \multicolumn{2}{c}{$\alpha=0.75$} & \multirow{2}{*}{\centering $M$}  & \multicolumn{2}{c}{$\alpha=0.25$} & \multicolumn{2}{c}{$\alpha=0.5$} & \multicolumn{2}{c}{$\alpha=0.75$} \\
     \cmidrule{2-3}  \cmidrule{4-5}  \cmidrule{6-7}  \cmidrule{9-10}  \cmidrule{11-12}  \cmidrule{13-14}
      & $E_{\infty}(\tau,h)$ & $Order^{h}$ & $E_{\infty}(\tau,h)$ & $Order^{h}$ & $E_{\infty}(\tau,h)$ & $Order^{h}$ &  &  $E_{\infty}(\tau,h)$ & $Order^{h}$ & $E_{\infty}(\tau,h)$ & $Order^{h}$ & $E_{\infty}(\tau,h)$ & $Order^{h}$   \\
      \midrule
      4   & 9.565e-03 & *    &9.175e-03 & *    &8.763e-03 &*     &  4  &7.732e-02  &*     &5.880e-02 &*    &5.218e-02 &* \\
      8  & 6.348e-04 &3.913 &6.089e-04 &3.914 &5.836e-04 &3.909 &  8 &7.476e-03  &3.370 &6.898e-03 &3.091 &4.417e-03 &3.575 \\   
       16  & 4.156e-05 &3.933 &3.978e-05 & 3.947&3.897e-05 &3.904 &  16 &5.680e-04  &3.718 &4.522e-04 &3.931 &2.864e-04 &3.945 \\	
      32  & 2.614e-06 &3.991 &2.530e-06 & 3.975&1.403e-06 &3.946 &  32 &3.595e-05  &3.982 &2.798e-05 &4.015 &1.385e-05 &4.106 \\    
      \bottomrule              
    \end{tabular}}
\end{table}

\begin{figure}
 \centering
 \includegraphics[width=0.75\textwidth]{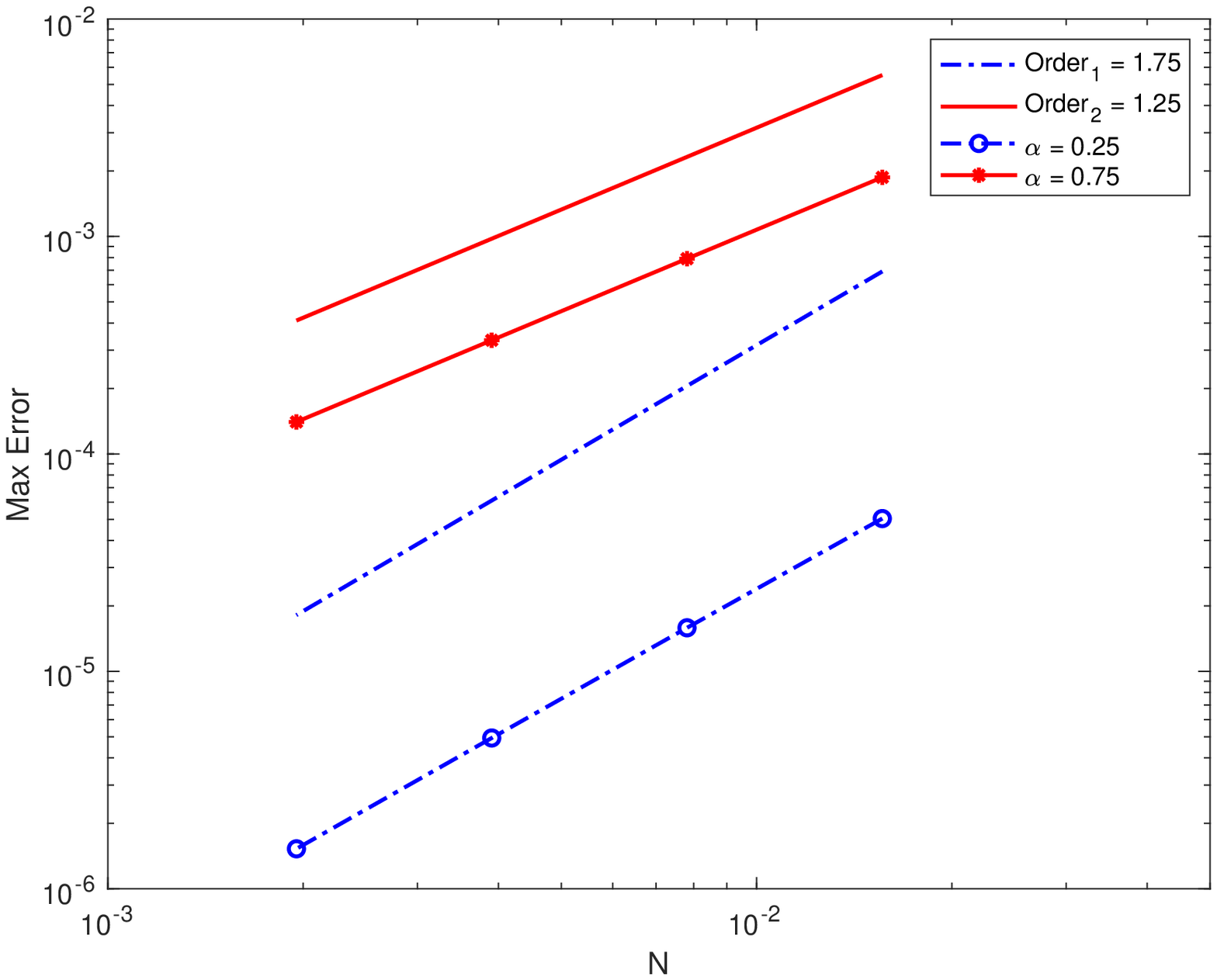}
\caption{Temporal convergence orders when $M=80$, $\mu_1=\mu_2=1$, and $\lambda=1$.}\label{fig1}
\end{figure}

\begin{figure}
    \centering
 \includegraphics[width=0.75\textwidth]{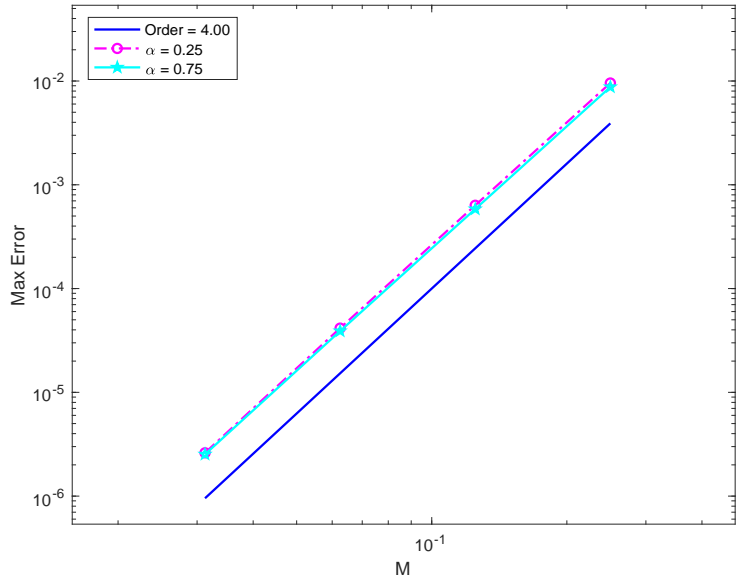}
 \caption{Spatial convergence orders when $N=20000$, $\mu_1=\mu_2=1$, and $\lambda=1$.}\label{fig2}
\end{figure}

\textbf{Example 2.} In the second example, we consider that the second-order derivative of time of the exact solution has weak singularity at zero. Although we see that $u(x,t)$ are supposed to belong to $C^2[0,T]$ in Lemma \ref*{lemma2.2}, the constructed compact difference scheme (\ref*{eqn3.10})-(\ref*{eqn3.13}) is still valid. Following, we
consider (\ref*{eq1.1}) at $\alpha=0.5$ and set $L=T=1$, and $\mu_1=\mu_2=1$, the exact solution is given as follows
$$u(x,t)=t^{\frac{3}{2}}\sin(2\pi x), $$
then the initial conditions are $\psi_1(x)=\psi_2(x)=0$, and the source term is
  $$f(x,t)=\frac{3\sqrt{\pi} \sin(2\pi x)}{4}\left( \mu_1+\mu_2 t\right)+\pi t^3 \sin(4\pi x)+4\lambda \pi^2t^{\frac{3}{2}}\sin(2\pi x) .$$
  \vskip 0.2mm
  In Tables 3 and 4, the numerical results show that the compact difference scheme (\ref*{eqn3.10})-(\ref*{eqn3.13}) is still applicable for some functions whose second-order derivative of time has weak singularity. The corresponding convergence orders are same as that in Example 1. In other words, the convergence orders are 4 for space and ($2-\alpha$) for time in $L^{\infty}$ norm. 

  \begin{table}\label{tb3}
    \center
    \caption{The $L^{\infty}$-norm errors and temporal convergence orders in viscous coefficients $\lambda=1$ and $0.01$, respectively.}
    \begin{tabular}{ccccccc} 
      \toprule
      \multicolumn{7}{c}{Compact difference scheme (\ref*{eqn3.10})-(\ref*{eqn3.13}) with $M=80$}\\
     \midrule
     & \multicolumn{2}{c}{$\lambda=1$} & &\multicolumn{2}{c}{$\lambda=0.01$}& \\
     \cmidrule{2-3}  \cmidrule{5-6}
      $N$ & $E_{\infty}(\tau,h)$ & $Order^{\tau}$ & $N$ & $E_{\infty}(\tau,h)$ & $Order^{\tau}$ &  \\
      \midrule
       64  &  3.832e-04  &  * & 64  &  7.567e-04  &  * &    \\
       128  &  1.515e-04  &  1.339 &128  &  2.815e-04  &  1.427 &    \\
       256 &  5.840e-05  &  1.375  & 256 &  1.021e-04  &  1.474 &    \\
       512 &  2.199e-05  &  1.409    &512 &  3.577e-05  &  1.512 &    \\
       1024&  8.150e-06 &  1.432    &1024&   1.247e-05  &  1.520 &  \\
      \bottomrule              
    \end{tabular}
\end{table}

\begin{table}\label{tb4}
  \center
  \caption{The $L^{\infty}$-norm errors and spatial convergence orders in viscous coefficients $\lambda=1$ and $0.01$, respectively.}
  \begin{tabular}{ccccccc} 
    \toprule
    \multicolumn{7}{c}{Compact difference scheme (\ref*{eqn3.10})-(\ref*{eqn3.13}) with $N=20000$}\\
   \midrule
   &\multicolumn{2}{c}{$\lambda=1$} & & \multicolumn{2}{c}{$\lambda=0.01$}& \\
   \cmidrule{2-3}  \cmidrule{5-6}
    $M$ & $E_{\infty}(\tau,h)$ & $Order^{h}$ & $M$ & $E_{\infty}(\tau,h)$ & $Order^{h}$ &  \\
    \midrule
    8  &  2.067e-03  &  * &  8  &  1.506e-02  &  *& \\
    16  &  1.370e-04  &  3.916   & 16  &  1.047e-03  &  3.846&     \\
    32 &  8.915e-06  &   3.941  & 32 &  7.584e-05  &  3.787 &   \\
    64 & 5.590e-07   &  3.995   &64 &  4.619e-06  &  4.037  &   \\
    \bottomrule              
  \end{tabular}
\end{table}
 

     \vskip 0.2mm
     \textbf{Example 3.} In the third example, we consider the original problem (\ref*{eq1.1})-(\ref*{eq1.3}) with $L=T=1$, chooes the initial conditions $\psi_1(x)=x^2(x-L)^2(x^2-Lx+L^2)$, $\psi_2(x)=0$, and the source term $f(x,t)=0$. The exact solution is unknown.

     \vskip 0.2mm
    In this case, since the exact solution is unknown, we need to redefine the errors and convergence orders. First, denote the error and convergence order in time as follows
    $$F_{\infty}^1(\tau,h):= \max\limits_{ 0\leq i \leq M }\left| u_i^N(\tau,h)-u_i^{2N}(\tau/2,h) \right|,\qquad
    Order^\tau_1:=\log_2\left(\frac{F_{\infty}^1(2\tau,h)}{F_{\infty}^1(\tau,h)}\right). $$
    Then, we denote the error and convergence order in space as follows
    $$G_{\infty}^1(\tau,h):= \max\limits_{ 0\leq i \leq M }| u_i^N(\tau,h)-u_{2i}^{N}(\tau,h/2) |,\qquad
    Order^h_1:=\log_2\Big(\frac{G_{\infty}^1(\tau,2h)}{G_{\infty}^1(\tau,h)}\Big). $$
    \vskip 0.2mm
    Tables 5 and 6 list $L^{\infty}$-norm errors and corresponding convergence orders for compact difference scheme (\ref*{eqn3.10})-(\ref*{eqn3.13}). It can be observed with selected different $\alpha$, the spatial and temporal convergence orders approximate fourth order and order $2-\alpha$, respectively. Moreover, Figures \ref{fig3} and \ref{fig4} more visually show the spatial and temporal convergence with different values $\alpha$, when $\mu_1=\mu_2=1$, $\lambda=0.01$. These numerical results further verify the theories. 
    \begin{table}\label{tb5}
      \center
      \caption{The $L^{\infty}$-norm errors and temporal convergence orders in viscous coefficients $\lambda=1$ and $0.01$, respectively.}
      \resizebox{\textwidth}{!}{
      \begin{tabular}{cccccccccccccc} 
        \toprule
        \multicolumn{14}{c}{Compact difference scheme (\ref*{eqn3.10})-(\ref*{eqn3.13}) with $M=256$, $\mu_1=\mu_2=1$}\\
       \midrule
       \multicolumn{7}{c}{$\lambda=1$} & \multicolumn{7}{c}{$\lambda=0.01$} \\
       \cmidrule{2-7}  \cmidrule{9-14}
         \multirow{2}{*}{\centering $N$} & \multicolumn{2}{c}{$\alpha=0.05$} & \multicolumn{2}{c}{$\alpha=0.5$} & \multicolumn{2}{c}{$\alpha=0.85$} & \multirow{2}{*}{\centering $N$}  & \multicolumn{2}{c}{$\alpha=0.05$} & \multicolumn{2}{c}{$\alpha=0.5$} & \multicolumn{2}{c}{$\alpha=0.85$} \\
       \cmidrule{2-3}  \cmidrule{4-5}  \cmidrule{6-7}  \cmidrule{9-10}  \cmidrule{11-12}  \cmidrule{13-14}
        & $F_{\infty}^1(\tau,h)$ & $Order_1^{\tau}$ & $F_{\infty}^1(\tau,h)$ & $Order_1^{\tau}$ & $F_{\infty}^1(\tau,h)$ & $Order_1^{\tau}$ &  &  $F_{\infty}^1(\tau,h)$ & $Order_1^{\tau}$ & $F_{\infty}^1(\tau,h)$ & $Order_1^{\tau}$ & $F_{\infty}^1(\tau,h)$ & $Order_1^{\tau}$   \\
        \midrule
        16   & 2.868e-04 & *    &1.384e-04 & *    &8.506e-04 &*     &  32  &2.384e-07  &*     &2.094e-06 &*    &1.911e-05 &* \\
        32  & 7.184e-05 &1.997 &5.346e-05 &1.372 &4.070e-04 &1.063 &  64 &5.844e-08  &2.034 &7.857e-07 &1.414 &8.702e-06 &1.134 \\   
         64  & 1.751e-05 &2.037 &1.974e-05 & 1.437&1.903e-04 &1.097 &  128 &1.426e-08  &2.035 &2.891e-07 &1.442 &3.943e-06 &1.142 \\	
        128  & 4.478e-06 &1.967 &7.145e-06 & 1.466&8.681e-05 &1.132 &  256 &3.597e-09  &1.987 &1.051e-07 &1.460 &1.782e-06 &1.146 \\    
        \bottomrule              
      \end{tabular}}
  \end{table}

  \begin{table}\label{tb6}
    \center
    \caption{The $L^{\infty}$-norm errors and spatial convergence orders in viscous coefficients $\lambda=1$ and $0.01$, respectively.}
    \resizebox{\textwidth}{!}{
    \begin{tabular}{cccccccccccccc} 
      \toprule
      \multicolumn{14}{c}{Compact difference scheme (\ref*{eqn3.10})-(\ref*{eqn3.13}) with $N=1024$, $\mu_1=\mu_2=1$}\\
     \midrule
     \multicolumn{7}{c}{$\lambda=1$} & \multicolumn{7}{c}{$\lambda=0.01$} \\
     \cmidrule{2-7}  \cmidrule{9-14}
       \multirow{2}{*}{\centering $M$} & \multicolumn{2}{c}{$\alpha=0.25$} & \multicolumn{2}{c}{$\alpha=0.5$} & \multicolumn{2}{c}{$\alpha=0.75$} & \multirow{2}{*}{\centering $M$}  & \multicolumn{2}{c}{$\alpha=0.25$} & \multicolumn{2}{c}{$\alpha=0.5$} & \multicolumn{2}{c}{$\alpha=0.75$} \\
     \cmidrule{2-3}  \cmidrule{4-5}  \cmidrule{6-7}  \cmidrule{9-10}  \cmidrule{11-12}  \cmidrule{13-14}
      & $G_{\infty}^1(\tau,h)$ & $Order_1^{h}$ & $G_{\infty}^1(\tau,h)$ & $Order_1^{h}$ & $G_{\infty}^1(\tau,h)$ & $Order_1^{h}$ &  &  $G_{\infty}^1(\tau,h)$ & $Order_1^{h}$ & $G_{\infty}^1(\tau,h)$ & $Order_1^{h}$ & $G_{\infty}^1(\tau,h)$ & $Order_1^{h}$   \\
      \midrule
       8  & 1.194e-06 & *    &1.171e-06 & *    &6.954e-06&*     &  32  &5.077e-06  &*     &7.083e-07 &*    &4.690e-06 &* \\
       16 & 7.609e-08 &3.972 &7.314e-08 &4.000 &5.934e-07 &3.550 &  64 &3.489e-07  &3.863 &4.394e-08 &4.010 &3.346e-07 &3.809 \\   
       32  & 4.779e-09 &3.993 &4.571e-09 & 4.000&3.720e-08 &3.996 &  128 &2.209e-08  &3.981 &2.734e-09 &4.006 &2.596e-08 &3.691 \\	
       64 & 2.990e-10 &3.998 &2.857e-10 & 4.000&2.347e-09 &3.986 &  256 &1.390e-09  &3.990 &1.718e-10 &3.993 &1.703e-09 &3.928 \\    
      \bottomrule              
    \end{tabular}}
\end{table}
 
\begin{figure}
 \centering
 \includegraphics[width=0.75\textwidth]{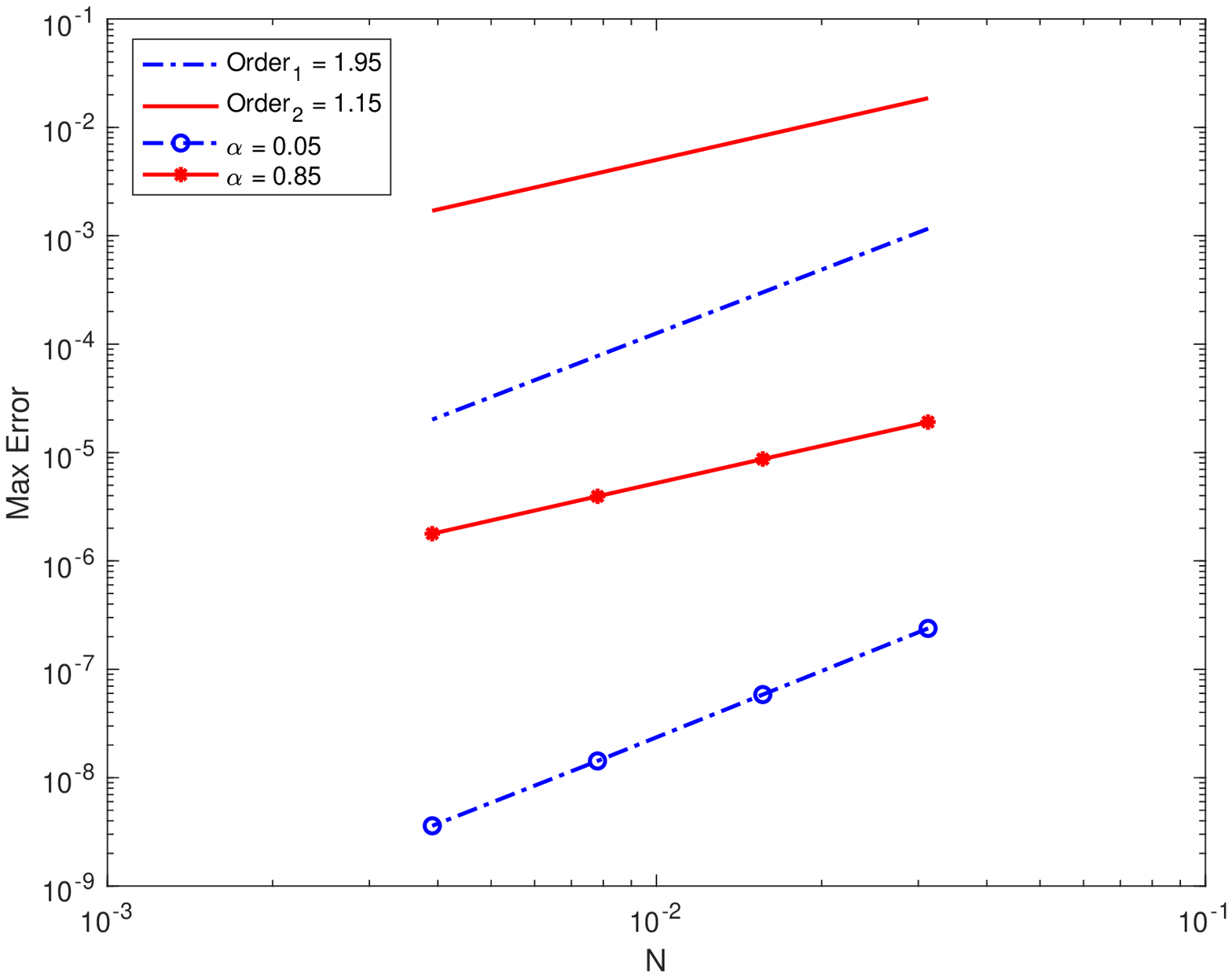}
\caption{Temporal convergence orders when $M=80$, $\mu_1=\mu_2=1$, and $\lambda=0.01$.}\label{fig3}
\end{figure}

\begin{figure}
    \centering
 \includegraphics[width=0.75\textwidth]{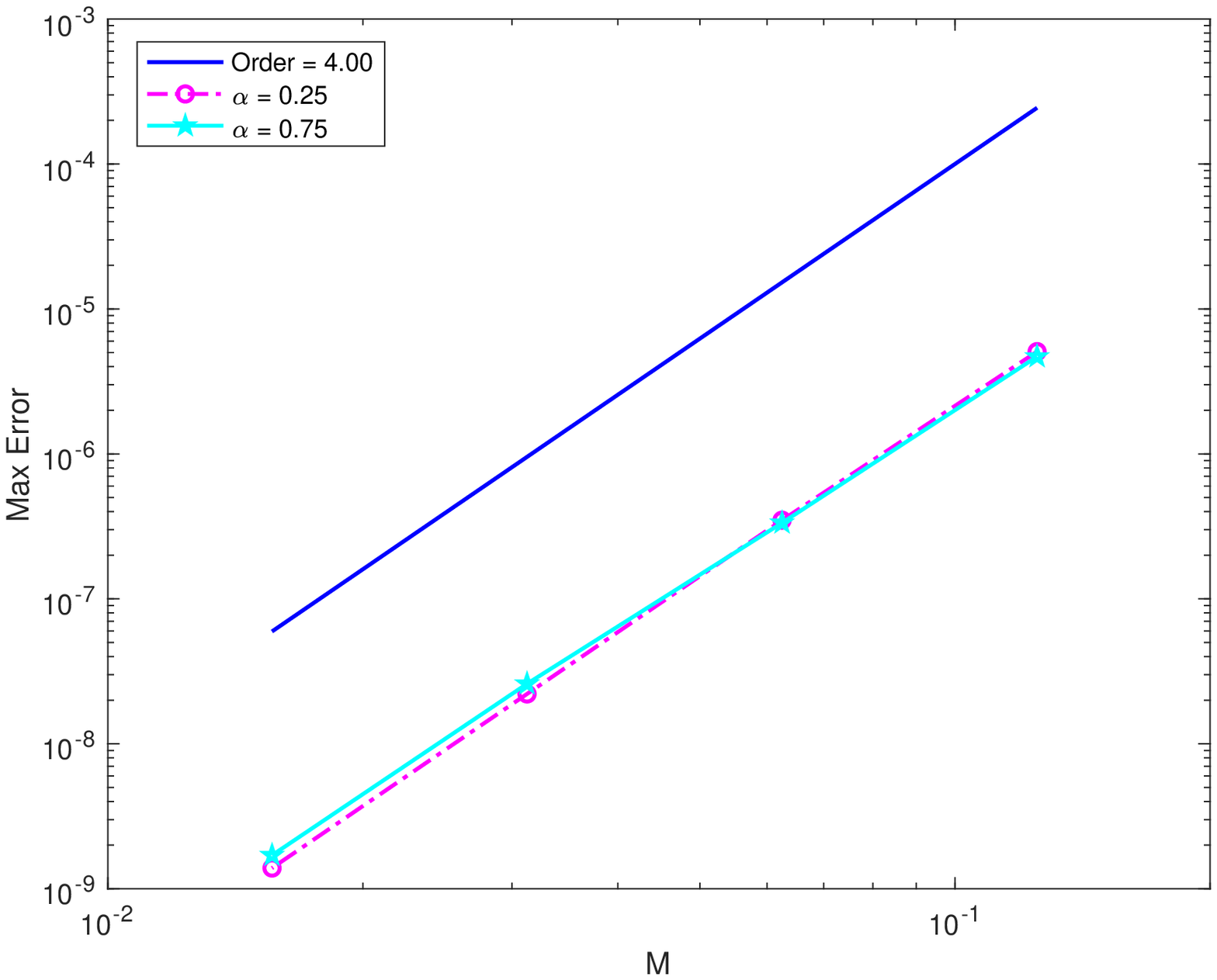}
 \caption{Spatial convergence orders when $N=1024$, $\mu_1=\mu_2=1$, and $\lambda=0.01$.}\label{fig4}
\end{figure}
     \section{Summary}
     In the current work, we have established a novel fourth-order compact difference scheme for mixed-type time-fractional Burgers' equation based on a developed nonlinear compact operator and reduction order technique. The convergence and stability in $L^{\infty}$-norm are deduced by discrete energy method. Three numerical examples illustrate the accuracy and effectiveness of the compact difference scheme. All numerical results are consistent with the theoretical analysis. In our future work, a generalized mixed-type time-fractional Burgers' equations will be considered by spatial compact difference method.

      \end{document}